\documentclass[12pt,reqno]{amsart} 
\usepackage{amssymb}
\usepackage{amsthm}
\usepackage{amsmath}
\usepackage{verbatim}
\usepackage{enumerate}
\usepackage[usenames,dvipsnames] {color}
\numberwithin{equation}{section}
\providecommand\ga{\gamma}
\newcommand\de{\delta}
\newcommand\lam{\lambda}

\renewcommand\phi{\varphi}
\def\d{\mathbb{D}}
\def\c{\mathbb{C}}
\def\t{\mathbb{T}}
\def\r{\mathbb{R}}

\newcommand\RHP{\mathbb{H}}
\newcommand\ov{\overline}
\newcommand\half{\tfrac 12}

\newcommand\inv{^{-1}}

\newcommand\bbm{\begin{bmatrix}}
\newcommand\ebm{\end{bmatrix}}
\newcommand\bpm{\begin{pmatrix}}
\newcommand\epm{\end{pmatrix}}

\def\sphere{\c \cup \{\infty\}}
\def\h{\mathcal{H}}
\def\l{\mathcal{L}}
\def\k{\mathcal{K}}

\def\P{\mathcal{P}}
\def\I{\mathcal{I}}
\def\b{\mathcal{B}}
\def\m{\mathcal{M}}
\def\n{\mathcal{N}}

\def\w{\mathcal{W}}
\def\ph{\phi}
\def\f{\mathcal{F}}

\newcommand\calh{\mathcal{H}}
\newcommand\df{\stackrel{\rm def}{=}}

\newcommand\eps{\varepsilon}
\def\be{\begin{equation}}
\def\ee{\end{equation}}

\def\hol{{\rm Hol}}
\def\hinf{{\rm H}^\infty}

\def\set#1#2{\{ #1 \, : \, #2\}}

\def\norm#1{\| #1 \|}
\def\bignorm#1{\big\| #1  \big\|}
\def\ip#1#2{\langle #1,#2\rangle}

\def\s0{s_0}
\def\p0{p_0}

\DeclareMathOperator{\ran}{ran}

\DeclareMathOperator{\re}{\rm Re}

\newcommand{\twopartdef}[4]
{
	\left\{
		\begin{array}{ll}
			#1 & \mbox{if } #2 \\
			#3 & \mbox{if } #4
		\end{array}
	\right.
}
\newcommand\forkde{for $K_\de $ }

\theoremstyle{definition}
\newtheorem{defin}[equation]{Definition}
\newtheorem{lem}[equation]{Lemma}
\newtheorem{prop}[equation]{Proposition}
\newtheorem{cor}[equation]{Corollary}
\newtheorem{thm}[equation]{Theorem}

\newtheorem{fact}[equation]{Fact}
\newtheorem{rem}[equation]{Remark}
\newtheorem{problem}[equation]{Problem}
\newtheorem{ex}[equation]{Example}
\begin{document}

\title[On the Operators with Numerical Range in an Ellipse]{On the Operators with Numerical Range in an Ellipse}

\author{Jim Agler}
\address{Department of Mathematics, University of California at San Diego, CA \textup{92103}, USA}

\email{jagler@ucsd.edu}

\author{Zinaida A. Lykova}
\address{School of Mathematics,  Statistics and Physics, Newcastle University, Newcastle upon Tyne
	NE\textup{1} \textup{7}RU, U.K.}

\email{Zinaida.Lykova@ncl.ac.uk}

\author{N. J. Young}
\address{School of Mathematics, Statistics and Physics, Newcastle University, Newcastle upon Tyne NE1 7RU, U.K.}

\email{Nicholas.Young@ncl.ac.uk}

\date{5th June, 2024}



\thanks{Partially supported by National Science Foundation Grants
	DMS 1361720 and 1665260, Heilbronn Institute for Mathematical Research (Focused Research Grant) and the Engineering and Physical Sciences Research Council grant EP/N03242X/1. }

\begin{abstract} 
   We give new necessary and sufficient  conditions for the numerical range $W(T)$ of an operator $T \in  \mathcal{B}(\mathcal{H})$  to be a subset of the closed elliptical set $K_\de\subseteq \mathbb{C}$  given by
\[
K_\delta {\stackrel{\rm def}{=}}
 \left\{x+iy: \frac{x^2}{(1+\delta)^2} + \frac{y^2}{(1-\delta)^2} \leq 1\right\},
\] 
where  $0 < \delta < 1$. Here $\mathcal{B}(\mathcal{H})$ denotes the collection of bounded linear operators on a Hilbert space $\mathcal{H}$.
Central to our efforts is a direct generalization of Berger's well-known criterion for an operator to have numerical radius at most one, his so-called {\em strange dilation theorem}. 
Specifically, we show that, if $T$
acts on a finite-dimensional Hilbert space $\mathcal{H}$ and satisfies a certain genericity assumption, then $W(T) \subseteq K_\delta$  if and only if there exists a Hilbert space $\mathcal{K} \supseteq \mathcal{H}$, 
operators $X_1$ and $X_2$ on $\mathcal{H}$
and a unitary operator $U$ acting on $\mathcal{K}$ such that
\begin{equation}
X_1+X_2=T, \quad X_1X_2=\delta
\end{equation}
and
\begin{equation}
X_1^k + X_2^k = 2P_\mathcal{H} U^k\left|_{\mathcal{H}}, \right. \;  k = 1,2, \dots,
\end{equation}
where  $P_\mathcal{H}$ denotes the orthogonal projection from $\mathcal{K}$ to $\mathcal{H}$.
 
We next generalize the lemma of Sarason that describes power dilations in terms of semi-invariant subspaces to operators $T$ that satisfy the relations (0.1) and (0.2).  This generalization yields a  characterization of the operators $T\in \mathcal{B}(\mathcal{H})$ such that $W(T)$ is contained in $K_\delta$ in terms of certain structured contractions that act on $\mathcal{H} \oplus \mathcal{H}$. 

 As a corollary of our results we extend Ando's parametrization
 of operators having numerical range in a disc to those $T$ such that $W(T)\subseteq K_\delta$. We prove that, if $T$ acts on a finite-dimensional Hilbert space $\mathcal{H}$, then $W(T)\subseteq K_\delta$ if and only if there exist a pair of contractions $A,B \in \mathcal{B}(\mathcal{H})$ such that $A$ is self-adjoint and
\[
T=2\sqrt\delta A + (1-\delta)\sqrt{{1+A}}\ B\sqrt{{1-A}}.
\]
 We also obtain 
 a formula for the B. and F. Delyon calcular norm of an analytic function on the inside of an ellipse in terms of the extremal $H^\infty$-extension problem  for analytic functions defined on a slice of the symmetrized bidisc.
\end{abstract} 

\maketitle
\tableofcontents

\section{Introduction}
In this section we shall introduce notation, recall some classical results from the literature and describe the results of the paper.

For $\h$ a complex Hilbert space, $\b(\h)$ will denote the collection of bounded linear operators on $\h$.  For $T\in\b(\h)$, $\sigma(T)$ will denote the spectrum of $T$, and $W(T)$ the \emph{ numerical range} or {\em field of values} of $T$, which is defined by the formula
\[
W(T) = \set{\ip{Tu}{u}}{u\in\h, \norm{u}=1 }.
\]
Accounts of the well-established theory of the numerical range of a bounded linear operator can be found in the books \cite{GuRao97,WuGau}. 
 
    In this paper, for fixed $\de \in [0,1)$,  we study operators whose numerical ranges lie in the closed subset $K_\de$ of the complex plane $\c$ bounded by the ellipse with major axis $[-1-\de,1+\de]$ and minor axis $i[-1+\de,1-\de]$, so that
\be    \label{defKde}   
K_\de \df  \left\{x+iy: x,y\in \r, \frac{x^2}{(1+\de)^2} + \frac{y^2}{(1-\de)^2} \leq 1\right\}.
\ee
Note that $K_0= \d^-$, the closure of the unit disc $\d=\set{z\in\c}{|z|<1}$ in $\c$. In this case
there is already a well-developed theory, due to Berger \cite{berg}, Ando \cite{Ando}, Kato \cite{kato}, Dritschel and Woerdeman \cite{dw97} and others, of operators whose numerical ranges are contained in $\d^-$.  \\

 \begin{center}{\sc 1.1.    Berger's strange dilation theorem} \label{genBerger} \end{center}
In \cite{berg} Berger proved the following theorem.
\begin{thm}\label{int.thm.10} {\bf (Berger's Theorem)} \,
If $T\in \b(\h)$, then $W(T) \subseteq \d^-$ if and only if there exists a Hilbert space $\k \supseteq \h$ and a unitary operator $U\in \b(\k)$ such that
\be\label{int.10}
T^k = 2P_\h\ U^k \Big|\ \h, \qquad k=1,2,\dots,
\ee
where  $P_\mathcal{H}$ denotes the orthogonal projection from $\mathcal{K}$ to $\mathcal{H}$.
\end{thm}
As the Sz-Nagy Dilation Theorem \cite{szn53} asserted that if $T\in \b(\h)$, then $\norm{T} \le 1$ if and only if there exists a Hilbert space $\k \supseteq \h$ and a unitary operator $U\in \b(\k)$ such that
\be\label{int.20}
T^k = P_\h\ U^k \Big|\ \h, \qquad k=1,2,\dots,
\ee
Berger referred to his theorem as a \emph{strange dilation theorem}. 

Here,  we shall show how Berger's theorem can be generalized to give conditions for the numerical range of an operator to be a subset of the elliptical set $K_\de$.  
To do this, we reinterpret the formulae \eqref{int.10} as a representation theorem for a certain rational $\b(\h)$-valued function defined in terms of $T$. Firstly, observe that if in Theorem \ref{int.thm.10} we relax the condition that $\h$ be a subspace of $\k$ to the condition that $\h$ merely be identified with a subspace of $\k$ via a Hilbert space isometry $I$, then the formula \eqref{int.10} takes the form
\be\label{int.30}
T^k = 2I^*U^k I, \qquad k=1,2,\dots.
\ee
Secondly, observe that the formulae \eqref{int.30} imply that,  for all $z \in \d$, and for any unitary operator $U$,
\begin{align*}
 I^*\frac{2}{1-zU} I &=2+\sum_{k=1}^\infty z^k\, 2I^*U^kI\\
&=2+\sum_{k=1}^ \infty z^k\, T^k= \frac{2-zT}{1-zT}.
\end{align*}
The above observations give us an alternative formulation of Berger's strange dilation theorem.
\begin{thm}\label{int.thm.20}  {\bf (Reformulation of Berger's Theorem)}
Let $T\in \b(\h)$ be an operator satisfying $\sigma(T) \subseteq \d^-$.
Then $W(T) \subseteq \d^-$ if and only if there exists a Hilbert space $\k$,
an isometry $I:\h \to \k$, and a unitary operator $U\in \b(\k)$ such that
\be\label{int.50}
\frac{1-\half zT}{1-zT} = I^*\frac{1}{1-zU} I \quad \text{ for all }z\in\d.
\ee
\end{thm}

\begin{center}{\sc 1.2.   A generalization of Berger's strange dilation theorem} \label{sub-sec-genBerger} \end{center}

The following result, which we prove in Section \ref{dilation-sec}, is a Berger theorem on an ellipse, modelled on the alternative formulation in Theorem \ref{int.thm.20}. 

\begin{thm} \label{morestrange} For  an operator $T\in \b(\h)$ and for any $\delta\in [0,1),$ the numerical range $W(T)$ is contained in  the closed set $K_\de$
 if and only if there exists a Hilbert space $\k$, an isometry $I:\h \to \k$, and a unitary operator $U\in \b(\k)$ such that
\be\label{ell.30.intro1}
\frac{1-\tfrac12 z T}{1-zT +\delta z^2} = I^*\frac{1}{1-zU} I \quad \text{ for all }z\in\d.
\ee
\end{thm}
Theorem \ref{morestrange} prompts us to  say that, for an operator $T\in \b(\h)$ satisfying $\sigma(T) \subseteq K_\de$, a  triple $(\k,I,U)$ is a {\em strange dilation  of $T$ relative to $K_\de$} if $\k$ is a Hilbert space, $I:\h\to\k$ is an isometry, $U$ is a unitary operator on $\k$ and the formula \eqref{ell.30.intro1} holds.

For any $\delta\in (0,1)$, an alternative expression of the notion of a strange dilation of $T$ relative to $K_\de$ is in terms of a pair $(X^1,X^2)$ of operators such that 
\[
\sigma(X^1)\cap\sigma(X^2)=\emptyset, \, X^1+X^2=T \; \text{ and}\; X^1X^2=\de.
\]
  Such a pair exists whenever $T$ is generic\footnote{For $\de \in (0,1)$,
we say that $T\in \b(\h)$ is  \emph{generic \forkde} if $\dim \h$ is finite, $\sigma(T) \subseteq K_\delta$, $\sigma(T)$ consists of $\dim \h$ distinct points,  and $4\delta \not\in \sigma(T^2)$.} for $K_\de$.

Given such a pair $(X^1,X^2)$ we may express the left hand side of equation \eqref{ell.30.intro1} in the definition of a strange dilation via partial fractions, so that equation \eqref{ell.30.intro1} becomes 
\[
\frac{\half}{1-zX^1} + \frac{\half}{1-zX^2} = I^*\frac{1}{1-zU}I\quad \text{ for all }z\in\d.
\] 
This representation, when applied to equation \eqref{ell.30.intro1} in Theorem \ref{morestrange}, leads us to the statement that,  for an operator  $T$ which is  generic  for $K_\de$,  $W(T) \subseteq K_\de$ if and only if there exist a Hilbert space $\k \supseteq \h$, a unitary operator $U\in \b(\k)$ and a pair $(X^1,X^2)$ of operators on $\h$ such that $\sigma(X^1)\cap\sigma(X^2)=\emptyset, \, X^1+X^2=T$, $X^1X^2=\de$ and
\be \label{compress-int}
X_1^k + X_2^k = 2P_\h U^k\left|_{\h},  \right.\quad k=1,2, \dots,
\ee
which is an interesting modification of the relations \eqref{int.10}.\\

\begin{center}{\sc 1.3.  The Sarason segue and even stranger dilations } \label{segue} \end{center}

Consider the operators $T\in\b(\h)$ that satisfy a simpler condition than that of Theorem \ref{morestrange}, to wit, that there exists a Hilbert space $\k \supseteq \h$, an isometry $I:\h \to \k$, and a unitary operator $U\in \b(\k)$ such that
\[
\frac{1+ z T}{1-zT} = I^*\frac{1+zU}{1-zU} I \quad \text{ for all } z\in\d,
\]
or equivalently $T^k= I^*U^k I$ for all positive integers $k$.  According to the Nagy dilation theorem these operators $T$ are precisely the contractions.  A principal reason that the Nagy dilation theorem has had such an impact on operator theory was Sarason's discovery \cite{sar} of a geometric interpretation of the operator-theoretic notion of a dilation
in terms of semi-invariant subspaces  -- see the Appendix for a brief account of the theory. In Sections \ref{cutting}-\ref{characterization}  we amend Sarason's analysis to apply to equation \eqref{ell.30.intro1} in the case $\de>0$. 
  We are led to the notion of an ``even stranger dilation", where the unitary operator $U$ in equation \eqref{ell.30.intro1} is replaced by a contraction $Y$ acting on a {\em finite-dimensional} space. 
\begin{defin}\label{ell.def.intro}
Let $T\in \b(\h)$ where $\h$ is a finite-dimensional Hilbert space and assume that $\sigma(T) \subseteq K_\delta$. We say that a triple $(\l,E,Y)$ is \emph{an even  stranger dilation of $T$ relative to $K_\de$} if $\l$ is a Hilbert space with $\dim \l = 2\dim \h$, $E:\h \to \l$ is an isometry, $Y\in \b(\l)$ is a contraction,
and
\be\label{ell.50.intro}
\frac{1-\frac{1}{2}  zT}{1-zT +\delta z^2}=E^*\frac{1}{1-zY}E \quad \text{ for all }z\in\d.
\ee
\end{defin}

In Section \ref{cutting}  we prove the following proposition (Proposition \ref{ell.prop.30} below).
\begin{prop}\label{ell.prop.30.intro}  Let $\de\in (0,1)$,
let $T\in \b(\h)$, let $\dim \h$ be finite. Assume that $T$ is generic  for $K_\de$. Then $W(T)\subseteq K_\de$
 if and only if $T$ has an even stranger dilation relative to $K_\de$.
\end{prop}

The first dividend of this proposition and Proposition \ref{ell.prop.50} below is the following characterization of  operators $T$ such that $W(T)\subseteq K_\de$. 

\begin{thm}\label{par.thm.10.intro} Let $\de \in (0,1)$.
Let $\dim\h$ be finite and let $T\in\b(\h)$  be generic for $K_\de$.
Then the following conditions are equivalent.
\begin{enumerate}[(i)]
\item $W(T)\subseteq K_\de$. 
\item For some (equivalently for all) choice of $Q$ commuting with $T$ and satisfying $Q^2=T^2-4\delta$, there exists an invertible $S \in \b(\h)$ such that
\be\label{par.10.intro}
\bignorm{\frac12\ \begin{bmatrix} T&QS\inv\\
SQ& STS\inv\end{bmatrix}}
 \le 1.
\ee
\end{enumerate}
\end{thm}  
Theorem \ref{par.thm.10.intro} is central to subsequent results of the paper.

\begin{center}{\sc 1.4. Germinators and the geometric structure of even stranger dilations}\label{evenstranger} \end{center}

Theorem \ref{par.thm.10.intro}  enables us to prove in Theorem \ref{thm.10} that, for generic $T$  with $W(T)\subseteq K_\de$, $T$ can be represented in the form
\be\label{formulais}
T=\left(C\oplus \de C\inv\right)|_{\calh\oplus\{0\}},
\ee
for some operator $C$ which is a {\em $(\delta,\h)$-germinator}, in the sense of the following definition.

\begin{defin}\label{par.def.10.intro}
Let $\de\in (0,1)$. If $\h$ is a finite-dimensional Hilbert space and $\dim \h=n$, we say that $C$ is a \emph{$(\delta,\h)$-germinator} if $C\in \b(\h\oplus \h)$ and $C$ satisfies the conditions:
\be\label{par.30.def.intro}
\norm{C} \le 1,
\ee
\be\label{par.40.def.intro}
C \text{ has } 2n \text{ distinct eigenvalues,}
\ee
\be\label{par.50.def.intro}
\frac{\delta}{C} =JCJ, \text{ where }  J=\begin{bmatrix}1&0\\
0& -1\end{bmatrix},
\ee
and
\be\label{par.55.def.intro}
0 \not\in \sigma(C_{12})
\ee
where $C_{12}$ is the 1-2 entry of $C$ when $C$ is represented as a $2\times 2$ block matrix acting on $\h \oplus \h$.
\end{defin}

\begin{center}\sc 1.5. An Ando theorem for the ellipse \end{center}

In Section \ref{parametrization} we exploit $(\de,\h)$-germinators, as defined in the previous subsection, to prove a generalization of Ando's theorem to the operators with numerical range in $K_\delta$.

We recall the relevant theorem of Ando \cite[Theorem 1]{Ando}.
\begin{thm} {\bf (Ando's theorem)}
For any bounded linear operator $T$ on a Hilbert space, the numerical range $W(T) \subseteq \d^-$ if and only if $T$ admits a factorization 
\[
T= \sqrt{1+A}~B\sqrt{1-A}
\]
for some self-adjoint contraction $A$ and some contraction $B$.
\end{thm}
Theorem \ref{Ando-ellipse} implies the following result.
\begin{thm} \label{1.10}
Let $\h$ be a finite dimensional Hilbert space, let $T\in \b(\h)$ and let $\delta\in [0,1)$. $W(T)\subseteq K_\delta$ if and only if there exist a pair of contractions $A,B \in \b(\h)$ such that
$A$ is self-adjoint and
\[
T=2\sqrt\delta A + (1-\delta)\sqrt{{1+A}}\ B\sqrt{{1-A}}.
\]
\end{thm}

\begin{center}\sc 1.6. A connection to Douglas-Paulsen operators  \end{center}

 In Section \ref{Douglas-Paulsen} we probe yet another application of Theorem 
\ref{par.thm.10.intro} and  of  $(\de,\h)$-germinators. 
For $\de\in [0,1)$ we adopt the notation
\begin{align}\label{defRde}
R_\de &\df  \left\{z\in\c: \de<|z|<1\right\}.
\end{align}
For any $\de\in(0,1)$ an operator $T$ on some Hilbert space $\h$ is said to be a {\em Douglas-Paulsen operator with parameter $\de$} if $\|T\|\leq 1$ and $\|T\inv\|\leq 1/\de$.   
 The name recognizes work of R. G. Douglas and V. Paulsen \cite{dp86}
which gave a dilation theory for Douglas-Paulsen operators.  An important step in their theory was the following estimate.
{\em If $X$ is a Douglas-Paulsen operator with parameter $\de$, $\sigma(X)\subseteq R_\de$ and $\phi$ is a bounded holomorphic matrix-valued function on $R_\de$ then}
\be\label{dps}
\|\phi(X)\| \leq \left(2+\frac{1+\de}{1-\de}   \right) \sup_{z\in R_\de} \|\phi(z)\|.
\ee
This result allowed them to show  that if $T \in \b(\h)$ is a Douglas-Paulsen operator, then there exists an invertible $S \in \b(\h)$ such that 
\be
\|S\| \|S^{-1}\|\leq \left(2+\frac{1+\de}{1-\de}   \right) 
\ee
and $S T S^{-1}$ dilates  to a normal operator  with spectrum contained in the boundary $\partial R_\de$. 

In the scalar case a slightly stronger result had been obtained earlier by A. Shields \cite[Proposition 23]{Shields}, with the smaller constant $2+\sqrt{\frac{1+\de}{1-\de}}$ on the right hand side.
He asked whether the constant $2+\sqrt{\frac{1+\de}{1-\de}}$  could be replaced by a quantity that remains bounded as $\de\to 1$.  This question was answered 
in the affirmative 
 by  C. Badea, B. Beckermann and M. Crouzeix  \cite{bbc2009}
and subsequently a better constant was established by
M. Crouzeix \cite{Crouzeix}.

It is relatively straightforward to show that if $X$ is a Douglas-Paulsen operator with parameter $\de$ then $W(X+\de X\inv)\subseteq K_\de$. However, not every operator with numerical range in $K_\de$ is itself of the form $X+\de X\inv$ for some Douglas-Paulsen operator $X$ with parameter $\de$ (see Fact \ref{10.3}).
 It is, though, the case that every operator $T$ such that $W(T)\subseteq K_\de$ can be {\em extended} to a Douglas-Paulsen operator with parameter $\de$, as is summarized in the following statement, which follows from
 Theorem \ref{conn-to-D-P-oper} in the body of the paper.
\begin{thm}\label{conn-to-D-P-oper-intro}
Let $\h$ be a finite-dimensional Hilbert space, let $T\in\b(\h)$ and let $\de\in(0,1)$.
$W(T)\subseteq K_\de$ if and only if there exists a Hilbert space $\k$ containing $\h$ and a Douglas-Paulsen operator $X \in  \b(\k)$ with parameter $\de$ such that $\h$ is invariant under $X+\de X\inv$ and $T$ is the restriction of $X+\de X\inv$ to $\h$.
\end{thm}

\begin{center}\sc 1.7. The B. and F. Delyon family for an ellipse  \end{center}
For $\de\in [0,1)$ we adopt the notation
\begin{align}\label{defGde}
G_\de &\df  \left\{x+iy: x,y\in \r, \frac{x^2}{(1+\de)^2} + \frac{y^2}{(1-\de)^2} < 1\right\}.
\end{align}
In Section \ref{norms-BFD-DP} we shall describe relations between holomorphic functions on the elliptical region\footnote{ $G_\de$ is the interior of $K_\de$ introduced previously.} 
$G_\de$ and on the annulus $R_\de$ that arise from the map $\pi: R_\de \to G_\de, \pi(z) = z+\de/z$, which maps $R_\de$ onto $G_\de$ in a 2-to-1 manner.

We recall from \cite[Chapter 9]{amy20} 
that  the {\em B. and F. Delyon family} $\f_{\mathrm {bfd}}(C)$ corresponding to
a bounded open convex set $C$ in $\c$  is the class of operators $T$ such that  the closure of the numerical range of $T$, $\overline{W(T)} \subseteq C$.  The nomenclature is a tribute to the ground-breaking theorem of the brothers B. and F. Delyon \cite{bfd1999}, which asserts that, for any $T\in\b(\h)$ such that  $\overline{W(T)} \subseteq C$  and any polynomial $p$,
\[
\|p(T)\| \leq \kappa(C) \sup_{z\in C} |p(z)|,
\]
where
\[
\kappa(C) = 3+ \left(\frac{2\pi(\mathrm{diam}(C))^2}{\mathrm{area}(C)}\right)^3.
\]
By \cite[Theorem 1.2-1]{GuRao97}, the spectrum $\sigma(T)$  of an operator $T$ is contained in $\overline{W(T)}$, and so, by the Riesz-Dunford functional calculus, $\ph(T)$ is defined for all $\phi\in \hol(C)$ and $T\in \f_{\mathrm {bfd}}(C)$. 
Recall further from \cite[Chapter 9]{amy20}, that
the {\em Douglas-Paulsen family } $\f_{\mathrm{dp}}(\de)$ corresponding to the annulus $R_\de$  is  the class of Douglas-Paulsen operators $X$ with parameter $\de$ 
 that satisfy the additional condition $ \sigma(X)\subseteq R_\de$. 
For $\Omega$ an open set in the plane, $\hol(\Omega)$ will denote the set of holomorphic functions defined on $\Omega$   and $\hinf(\Omega)$ will denote the set of bounded holomorphic functions defined on $\Omega$, with the supremum norm $\norm{\phi}_\infty= \sup_{z\in\Omega}|\phi(z)|$. 
 Following the setup of \cite[Chapter 9]{amy20}, we consider the calcular norms
\be \label{dpnorm}
\norm{\phi}_{\mathrm {dp}} = \sup_{X\in  \f_{\mathrm{dp}}(\de) }\norm{\phi(X)},
\ee
defined for $\phi \in \hol(R_\delta)$, and
\be \label{bfdnorm}
\norm{\phi}_{\mathrm {bfd}}=\sup_{T\in \f_{\mathrm {bfd}}(G_\de)}\norm{\phi(T)},
\ee
defined for $\phi \in \hol(G_\delta)$. 
There is no guarantee that the quantities defined by equations \eqref{dpnorm} and \eqref{bfdnorm} are finite.
Accordingly, we introduce the associated Banach algebras
\[
\hinf_{\mathrm {dp}}(R_\de)=\set{\phi \in \hol(R_\delta)}{\norm{\phi}_{\mathrm {dp}}<\infty}
\]
and
\[
\hinf_{\mathrm {bfd}}(G_\de)=\set{\phi \in \hol(G_\delta)}{\norm{\phi}_{\mathrm {\mathrm {bfd}}} <\infty}.
\]

The relationship described in Theorem \ref{conn-to-D-P-oper-intro}
 between the  B. and F. Delyon class corresponding to the elliptical region $G_\de$ and the Douglas-Paulsen operators with parameter $\de$ induces the following inequality
 \[
 \|\phi\|_{\mathrm {bfd}} \leq \|\phi\circ\pi\|_{\mathrm {dp}} \; \text{for all} \;
 \phi\in \hol(G_\de).
 \]
Surprisingly, Theorem \ref{dp.thm.10} from Section \ref{norms-BFD-DP} below shows that the reverse inequality also holds, so that
\[
\|\phi\|_{\mathrm {bfd}} = \|\phi\circ\pi\|_{\mathrm {dp}}   \quad \text{for all} \; \phi\in\hol(G_\delta).
\]

We shall say that a function $f\in\hol(R_\de)$ is {\em symmetric with respect to the involution $\lam\mapsto \de/\lam$ of $R_\de$} if $f(\lam)= f(\de/\lam)$ for all $\lam\in R_\de$.
By noting that a holomorphic function $f$ on $R_\de$ has the form $\phi\circ\pi$ for some $\phi\in\hol(R_\de)$ if and only if $f$ is symmetric with respect to the involution $z\mapsto \de/z$ on $R_\de$, we obtain  a precise relation between  the spaces $\hinf_{\mathrm {dp}}(R_\de)$ and $\hinf_{\mathrm {bfd}}(G_\de)$
(see Theorem \ref{dp.thm.10} below). 

\begin{thm}\label{dp.thm.10.intro} Let $\de\in(0,1)$.
The mapping $\phi  \mapsto \phi\circ\pi$ from
$ \hol(G_\delta)$ to $\hol(R_\delta)$
is an isometric isomorphism from $\hinf_{\mathrm {bfd}}(G_\de)$ onto the set
 of functions in $\hinf_{\mathrm {dp}}(R_\de)$ that are symmetric with respect to the involution $\lam\mapsto \de/\lam$ of $R_\de$.
\end{thm}

\begin{center}\sc 1.8.   An extremal problem for analytic functions on the symmetrized bidisc
 \end{center}

In the previous subsection we described results of Section \ref{norms-BFD-DP} which show how  the B. and F. Delyon norm $\|\cdot\|_{\mathrm{bfd}}$ can be interpreted in terms of function theory on $R_\de$.  
 In Section \ref{delyonnorm} we show that $\|\cdot\|_{\mathrm{bfd}}$  on $\hol(G_\de)$ can also be interpreted in terms of function theory on the symmetrized bidisc $G$,
the domain in $\c^2$ defined by
\begin{align*}
G & \df \{(z+w,zw):z,w\in\c, |z|<1,|w|<1\} \\
	&= \{(s,p)\in\c^2: |s-\bar s p| < 1-|p|^2\}.
\end{align*}
We observe that, for $\de\in (0,1)$ and $s\in\c$, a point $(s,\de)\in\c^2$ belongs to $G$ if and only if $s\in G_\de$.  Thus, if we are given a function $\phi\in\hol(G_\de)$, we could look to express $\phi$ as the restriction to $G_\de$ of a function $\Phi \in \hol(G)$.  Oka's Extension Theorem tells us there does exist $\Phi\in\hol(G)$ such that $\Phi(s,\de)=\phi(s)$ for all $s\in G_\de$.  We ask further, for which $\phi\in\hol(G_\de)$ can $\Phi$ be chosen bounded, and what is the minimal $H^\infty$ norm of all $\Phi$ that extend $\phi$?

The next theorem gives a full answer to these questions.
\begin{thm} \label{extremalproblem} Let $\de\in(0,1)$.  For any $\phi\in \hol(G_\de)$, the minimum of $\|\Phi\|_{H^\infty(G)}$ over all functions $\Phi \in H^\infty(G)$ that extend the function $(s,\de)\mapsto \phi(s)$ is $\norm{\phi}_{\mathrm {bfd}}$.
In particular, an analytic function $\phi$ on $G_\de$ has a bounded extension $\Phi$ to $G$ if
and only if $\phi\in H^\infty_{\mathrm{bfd}}(G_\de)$.\footnote{Note that, by the theorem of B. and F. Delyon, an analytic function $\phi$ on $G_\de$ belongs to $H^\infty_{\mathrm{bfd}}(G_\de)$ if and only if $\phi$ is bounded.  It follows that an analytic function $\phi$ on $G_\de$ has a bounded analytic extension $\Phi$ to $G$ if
and only if $\phi\in H^\infty(G_\de)$, though in general it can happen that the infimum of $\|\Phi\|_\infty$  over all such extensions $\Phi$ is strictly greater than $\|\phi\|_\infty$. }
\end{thm}

It follows from Theorem \ref{extremalproblem} that $\hinf_{\mathrm {bfd}}(G_\delta)$ is isometrically isomorphic as a Banach algebra to the quotient algebra $\hinf(G)$ modulo the ideal $\{\Phi\in\hinf(G):\Phi(s,\de)=0$ for all $s\in G_\de\}$, with the quotient norm.\\

\section{A condition for $W(T)$ to be in an ellipse}\label{ellipse}
Let $\delta \in [0,1)$.  We have introduced the open and closed elliptical sets $G_\delta$ and $K_\delta$ defined in equations \eqref{defGde} and \eqref{defKde}.
Their common boundary is
\[
\Gamma_\delta = K_\delta \setminus G_\delta = \set{x+iy\in \c}{\frac{x^2}{(1+\delta)^2} +\frac{y^2}{(1-\delta)^2} = 1}.
\]
We shall make extensive use of the Zhukovskii map, which relates an annulus to the elliptical region $G_\de$.  It is widely used in engineering applications, especially in aerodynamics, in the design of aerofoils \cite[page 677]{krey}.
\begin{lem}\label{ell.lem.10}
Let $\delta \in [0,1)$. If $f:\c \cup \{\infty\} \to \sphere$ is defined by
\[
f(z) = 1/z +\delta z,
\]
then
\begin{enumerate}[(i)]
\item $f$ maps $\d$ in a 1-1 manner onto $(\sphere) \setminus K_\delta$,
\item for $\delta \in (0,1)$, $f$ maps $A_\delta= \set{z \in \c}{1<|z|<1/\delta}$ in a 2\,-1 manner onto $G_\delta$,
 except for the points $\pm 2\sqrt{\de} \in G_\de$, whose pre-images in $A_\de$ are the singleton sets $\{\pm1/\sqrt{\de}\}$.
\item for $\delta \in (0,1)$, $f$ maps $\set{z \in \c}{|z|>1/\delta}$ in a 1\,-1 manner onto $(\sphere) \setminus K_\delta$, and
\item $f$ maps $\t$ onto $\Gamma_\delta$ and $f'(z) \ne 0$ for all $z \in \t$.
\end{enumerate}
\end{lem}
\begin{prop}\label{ell.prop.10} Let $\de \in [0,1)$.
Let $T\in \b(\h)$ with $\sigma(T) \subseteq K_\delta$. Then, for
$f$ as described in Lemma \ref{ell.lem.10},
\[
G(z) = \frac{zf'(z)}{T-f(z)}
\]
 is an analytic  $\b(\h)$-valued function on $\d\setminus\{0\}$ which extends to an analytic function on $\d$ satisfying $G(0)=1$. Furthermore,
 \be\label{ell.10}
 W(T) \subseteq K_\delta \iff \forall_{z\in \d}\ \re G(z) \ge 0.
 \ee
\end{prop}
\begin{proof}
By assumption, $\sigma(T) \subseteq K_\delta$.  By Lemma \ref{ell.lem.10} (i),
$f$ maps $\d\setminus\{0\}$ to the resolvent set of $T$.
Hence $(T-f(z))\inv$ is analytic on $\d\setminus\{0\}$.  Thus,
$G$ is analytic on $\d \setminus \{0\}$. Furthermore,
\[
\lim_{z \to 0} G(z) = \lim_{z \to 0} \frac{zf'(z)}{T-f(z)}= \lim_{z \to 0} \frac{\delta z^2-1}{zT-1- \delta z^2} =1,
\]
so the singularity of $G$ at 0 is removable, with $G(0)=1$.

Now observe that Lemma \ref{ell.lem.10} (iv) implies that for each $\theta$, $f(e^{i\theta})\in \Gamma_\delta$ and $ie^{i\theta}f'(e^{i\theta})$ points tangent to $\Gamma_\delta$ at $f(e^{i\theta})$. Furthermore, Lemma \ref{ell.lem.10} (i) implies that $ie^{i\theta}f'(e^{i\theta})$ is oriented in the clockwise direction. Consequently, for each $\theta$, $e^{i\theta}f'(e^{i\theta})$ is an inward pointing normal to $\Gamma_\delta$ at $f(e^{i\theta})$.

The observations in the preceding paragraph imply that for each $\theta$, the set $H_\theta$ of complex numbers defined by
\[
H_\theta = \set{\mu \in \c}{\re \overline{e^{i\theta}f'(e^{i\theta})}(\mu-f(e^{i\theta})) \ge 0},
\]
is the supporting half plane to $K_\delta$ at $f(e^{i\theta})$. Since $W(T) \subseteq K_\delta$ if and only if $W(T) \subseteq H_\theta$ for all $\theta$, and $W(T) \subseteq H_\theta$ if and only if
\[
\re \overline{e^{i\theta}f'(e^{i\theta})}(T-f(e^{i\theta})) \ge 0,
\]
it follows that
\be\label{ell.15}
 W(T) \subseteq K_\delta \iff \forall_{\theta\in \r}\ \re \overline{e^{i\theta}f'(e^{i\theta})}(T-f(e^{i\theta})) \ge 0.
\ee
Note
\[
 \overline{e^{i\theta}f'(e^{i\theta})}(T-f(e^{i\theta}))=
 |f'(e^{i\theta})|^2\frac{T-f(e^{i\theta})}{e^{i\theta}f'(e^{i\theta})}.
\]
If $X$ is an invertible operator, then 
\[
\re X^{-1}= \frac{1}{2} (X^{-1} + (X^{-1})^*)=X^{-1} \frac{1}{2}(X + X^*) (X^{-1})^*=  X^{-1} \re X (X^{-1})^*.
\]
Thus 
\be\label{ell.17}
\re X \ge 0 \ \text{if and only if}\  \re X^{-1} \ge 0. 
\ee
Hence, by the equivalences \eqref{ell.15} and \eqref{ell.17},
\begin{align}
 W(T) \subseteq K_\delta &\iff \forall_{\theta\in \r}\ \re \frac{e^{i\theta}f'(e^{i\theta})}{T-f(e^{i\theta})}\ge 0 \label{ell.20}\\
 	&\iff \re G(z) \geq 0  \ \text{for all}\ z\in\t \; \text{(by the definition of} \, G).   \label{ell.22}
\end{align}
  Finally, let us show that $\re G(z) \geq 0$ for all $z\in \d$ if and only if 
  $\re G(z) \geq 0$ for all $z\in \t$.
We have already shown that $G$
 is an analytic  $\b(\h)$-valued function on $\d$ and is continuous 
 on $\overline{\d}$.  It is immediate from the continuity of $G$ that if 
 $\re G(z) \geq 0$ for all $z\in \d$ then
  $\re G(z) \geq 0$ for all $z\in \d^-$, and so, in particular for all $z \in\t$.
Conversely,  suppose that, for all $z \in \t$, $\re G(z)\ge 0.$
Take the composition of the inverse $C^{-1}$ of the Cayley transform $C: \d \to \RHP$ with $G$, where $\RHP$ denotes the right half plane $\{z\in\c: \re z> 0\}$.
Then $ C^{-1} \circ G$ is an analytic $\b(\h)$-valued function on $\d$, continuous  on $\overline{\d}$ and maps $\t$ to the unit ball of $\b(\h)$. By the Maximum Principle, for all $z \in \d$,
\[ \| C^{-1} \circ G(z)\| \le 1.
\]
Hence, for all $z \in \d$, $\re G(z)\ge 0.$
The statement \eqref{ell.10} follows.
\end{proof}

\section{Strange dilation on an ellipse}\label{dilation-sec}

In this section we shall combine the  Herglotz Representation Theorem for operator-valued functions with Proposition \ref{ell.prop.10} to obtain an analog of Berger's Theorem for $K_\delta$.

We first recall the following generalization of the Herglotz Representation Theorem due to Naimark \cite{nai2,nai1}.
\begin{thm}\label{ell.thm.10}
Let $\h$ be a Hilbert space and assume that $V$ is an analytic $\b(\h)$-valued function defined on $\d$ satisfying $\re V(z) \ge 0$ for all $z\in \d$. Then there exist a Hilbert space $\k$, an isometry $I:\h \to \k$, and a unitary operator $U\in \b(\k)$ such that
\[
V(z) = I^*\frac{1+zU}{1-zU}I \quad \text{ for all }z\in\d.
\]
\end{thm}

We shall massage $G$ to produce a closely related function $F$, which has the virtue that it gives rise to a dilation theorem for operators with numerical range in an elliptical region that is closely analogous to Berger's strange dilation theorem (see Theorem \ref{int.thm.20} above).

\begin{lem}\label{ell.prop.101} Let $\de \in [0,1)$.
Let $T\in \b(\h)$ with $\sigma(T) \subseteq K_\delta$. Then
\be\label{ell.25}
F(z) = \frac{1-\tfrac12 z T}{1-zT +\delta z^2},\qquad z\in \d,
\ee
is  an analytic  $\b(\h)$-valued function on $\d$ satisfying $F(0)=1$. 
Moreover, for every $z \in \d$,  
\be\label{ell.250}
F(z)=\frac{1}{2}\ (G(z) +1),
\ee
where 
\[
G(z) = \frac{zf'(z)}{T-f(z)}\;\; \text{and} \;\; f(z) = 1/z +\delta z.
\]
\end{lem}
\begin{proof}
Note that, for all $z\neq 0$,
\[
1-zT +\delta z^2 = z(\frac{1}{z} +\delta z -T) =z (f(z)-T).
\]
Hence, for $z\in\d\setminus \{0\}$,
\begin{align*}
\frac{1}{2}\ (G(z) +1) &= \frac{1}{2}\ \left( \frac{zf'(z)}{T-f(z)} +1 \right) 
= \frac{1}{2}\ \left( \frac{z( -1/z^2 + \delta)}{T-f(z)} +1 \right)\\
&= \frac{1}{2}\ \left(\frac{-1/z + \delta z - 1/z -\delta z +T )}{T-f(z)}\right)
= \frac{1}{2}\ \left(\frac{-2/z +T }{T-f(z)}\right)\\
&= \frac{1 - \frac{1}{2}zT}{z(f(z)-T)}= F(z).
\end{align*}
By  Proposition  \ref{ell.prop.10}, $G$ is an analytic $\b(\h)$-valued function on $\d\setminus \{0\}$, with a removable singularity at $0$, satisfying $G(0)=1$. 
Hence, $F$ is a well defined analytic function on $\d\setminus\{0\}$, with a removable singularity at $0$, satisfying $F(0)=1$. 
\end{proof}

We shall now prove Theorem \ref{morestrange} from the Introduction.
\begin{thm}\label{ell.prop.20} Let $\de \in [0,1)$.
Let $T\in \b(\h)$ be an operator satisfying $\sigma(T) \subseteq K_\delta$, 
and let $F:\d \to \b(\h)$ be defined by the formula
\[
F(z) = \frac{1-\tfrac12 z T}{1-zT +\delta z^2} \qquad \text{ for }  \; z\in \d.
\]
Then $W(T) \subseteq K_\delta$ if and only if there exists a Hilbert space $\k \supseteq \h$, an isometry $I:\h \to \k$, and a unitary operator $U\in \b(\k)$ such that
\be\label{ell.30}
F(z) = I^*\frac{1}{1-zU} I \quad \text{ for all }z\in\d.
\ee
\end{thm}
\begin{proof}
Proposition \ref{ell.prop.10} and Theorem \ref{ell.thm.10} imply that $W(T) \subseteq K_\delta$ if and only if,  for  the analytic function $G:\d \to \b(\h)$ defined by
\[
G(z) = \frac{zf'(z)}{T-f(z)}, \quad z \in \d,
\]  
there exists a Hilbert space $\k$, an isometry $I:\h \to \k$, and a unitary operator $U\in \b(\k)$ such that
\be\label{ell.300}
G(z) = I^*\frac{1+zU}{1-zU}I \quad \text{ for all }z\in\d.
\ee 

 Note that, for all $z \in \d$, 
\[
\frac{1}{2}\ \left(\frac{1+zU}{1-zU}+1\right) = \frac{1}{1-zU}.
\]
Hence, for all $z \in \d$,
\[
I^*\frac{1}{1-zU}I = I^* \frac{1}{2}\ \left(\frac{1+zU}{1-zU}+1\right) I
 = \frac{1}{2}\ \left(I^*\frac{1+zU}{1-zU}I +1\right). 
\]
By Lemma \ref{ell.prop.101}, for all $z \in \d$,
\[
 F(z)= \frac{1}{2}\ (G(z) +1).
\]
Thus 
\[
G(z) = I^*\frac{1+zU}{1-zU}I \;\; \text{for all} \; z \in \d,
\] 
if and only if 
\[
F(z) = I^*\frac{1}{1-zU} I \;\; \text{for all} \; z \in \d.
\]
Therefore, $W(T) \subseteq K_\delta$ if and only if there exists a Hilbert space $\k \supseteq \h$, an isometry $I:\h \to \k$, and a unitary operator $U\in \b(\k)$ such that
\[
F(z) = I^*\frac{1}{1-zU} I  \quad \text{ for all }z\in\d.
\]
\end{proof}

In light of the obvious similarity of Theorem \ref{int.thm.20} and Theorem \ref{ell.prop.20}, and  in honor of Berger's seminal contribution \cite{berg} we make the following definition.
\begin{defin}\label{ell.def.10}
Let $\delta \in [0,1)$, and let $T\in \b(\h)$ be an operator satisfying $\sigma(T) \subseteq K_\delta$. We say that a triple $(\k, I, U)$ is a \emph{strange dilation of $T$  relative to $K_\de$ } if $\k$ is a Hilbert space, $I:\h \to \k$ is an isometry, $U\in \b(\k)$ is unitary, and
the formula 
\be\label{ell.30.sd}
\frac{1-\tfrac12 z T}{1-zT +\delta z^2}= I^*\frac{1}{1-zU} I \quad \text{ holds for all }z\in\d.
\ee
\end{defin}
Where the role of $K_\de$ is clear from the context, we will omit ``relative to $K_\de$" when speaking of strange dilations.

\begin{prop} \label{str-dil-inf}
Let $\de \in (0,1)$, let $\mu\in G_\de$, and let  $T$ be the scalar operator of multiplication by $\mu$ on a (finite- or infinite-dimensional) Hilbert space $\h$. Then, for any  strange dilation $(\k,I,U)$ of $T$,   $\k$ has infinite dimension. 
\end{prop}
\begin{proof} 
Assume that $T$ has the strange dilation $(\k,I,U)$, so that
\be\label{dilation}
\frac{1-\half T z}{1-T z+\de z^2} =I^*\frac{1}{1-zU}I \; \text{for all}\; z\in\d.
\ee
To see that $\k$ must be infinite-dimensional, suppose to the contray that $\dim \k =n < \infty$. Then there is an orthonormal basis $e_1,\dots, e_n$ of $\k$ consisting of eigenvectors of $U$, with corresponding eigenvalues $\tau_1,\dots,\tau_n$, all of which have unit modulus.
Let $x$ be a nonzero vector in $\h$. We have
\[
 \frac{1-\half\mu z}{1-\mu z+\de z^2}x= \frac{1-\half T z}{1-T z+\de z^2}x=I^*\frac{1}{1-zU}Ix.
\]
Write $x_i=\ip{Ix}{e_i}_\k$ for $i=1,\dots,n$, so that $Ix= \sum_{i=1}^n x_i e_i$.
Observe that since $x\neq 0$ and $I$ is an isometry, the components $x_1,\dots,x_n$ of $Ix$ are not all zero.  Choose $j\in\{1,\dots,n\}$ such that $x_j\neq 0$ and notice that $I^*e_j \neq 0$, since $0\neq x_j=\ip{x}{I^*e_j}$.  We have
\begin{align}\label{inf-dil}
\frac{1-\half\mu z}{1-\mu z+\de z^2}x &=I^*\frac{1}{1-zU}\sum_{i=1}^n x_i e_i \notag\\
&=I^*\sum_{i=1}^nx_i (1-zU)\inv e_i  =\sum_{i=1}^nx_i (1-z\tau_i)\inv I^*e_i.
\end{align}
 By Lemma \ref{ell.prop.101} the rational function $F(z) \df \frac{1-\half\mu z}{1-\mu z+\de z^2}$ is analytic on $\d$. One can check that, when $\mu\in G_\de$, the quadratic
$1-\mu z+\de z^2$ has no zeros on the unit circle $\t$. Thus
 $F(z)$ is analytic on $\d^-$.
As $z \to \bar\tau_j$ in equation \eqref{inf-dil}
the left hand side tends to a limit in $\h$, while the right hand side is unbounded.  This contradiction shows that $\k$ is not finite-dimensional.
\end{proof}

\section{Prepairs for generic operators $T$  and residues}
Let $\de \in (0,1)$, and let
\[
R_\delta = \set{z\in \c}{\delta <|z| <1}.
\]
Note that if we define $\pi:\sphere \to \sphere$ by the formula
\[
\pi(\lambda) = \lambda+\frac{\delta}{\lambda},
\]
then as $\pi(\lambda)=f(1/\lambda)$, Lemma \ref{ell.lem.10} implies that $\pi:R_\delta \to G_\delta$ is a 2-1 cover. Also, for each $\mu \in G_\delta$, all solutions to the equation $\pi(\lambda)=\mu$ lie in $R_\delta$.
\begin{defin}\label{ell.def.20}  Let $\de\in(0,1)$.
We say that $T\in \b(\h)$ is  \emph{generic \forkde}  if $\dim \h$ is finite, $\sigma(T) \subseteq K_\delta$,
$\sigma(T)$ consists of $\dim \h$ distinct points,  and for each $\mu \in \sigma(T)$,  the equation $\pi(\lambda) =\mu$ has two distinct solutions in $R_\delta$ (equivalently, $4\delta \not\in \sigma(T^2)$).
\end{defin}
If $T\in \b(\h)$ is generic \forkde and $n=\dim \h$, then there exist $n$ distinct points $\mu_1,\ldots,\mu_n \in \c$ (the eigenvalues of $T$) and $n$ vectors $e_1,\ldots, e_n \in \h$ (the corresponding eigenvectors of $T$) such that
\be\label{ell.32}
Te_i = \mu_i e_i\qquad\mbox{ for } i=1,\ldots,n.
\ee
Once such an enumeration of the eigenvalues of $T$ is chosen, then $T$ is generic precisely when the equation $\pi(\lambda) = \mu_i$ has two distinct solutions for each $i$. If for each $i$, we let $\lambda_i^1,\lambda_i^2$ be an enumeration of these two solutions, then we may define a pair of operators $X^1, X^2 \in \b(\h)$ by the formulas
\be\label{ell.33}
X^r e_i = \lambda_i^r e_i  \qquad \mbox{ for } r=1,2\qquad \mbox{ and } i=1,\ldots,n.
\ee
Evidently, if $X^1$ and $X^2$ are so defined, then

\be\label{34}
\sigma(X^1) \cap \sigma(X^2) = \varnothing,
\ee
\be\label{35}
X^1+X^2 = T, 
\ee
\be\label{36}  \text{and}\;
X^1X^2=\delta.
\ee
\begin{defin}\label{ell.def.30}  Let $\de\in(0,1)$.
Let $T\in \b(\h)$  be generic \forkde and let $\dim \h =n$. We say that a pair $X=(X^1,X^2)$ where $X^1,X^2 \in \b(\h)$ is a \emph{prepair} for $T$ \emph{relative to $K_\de$} if
\begin{enumerate}[(i)]
\item $\sigma(X^1) \cap \sigma(X^2) = \varnothing,$
\item $X^1+X^2 = T, \;\; \text{and}$
\item $X^1X^2=\delta$.
\end{enumerate}
We say that $X$ is a \emph{generic prepair} for $T$ if, in addition, both $X^1$ and $X^2$
have $n$ distinct eigenvalues. 
\end{defin}
\begin{rem}  Let $T\in \b(\h)$  be generic \forkde ~and let $X=(X^1,X^2)$ be a generic prepair for $T$.
It is easy to see that 
\be\label{TX1X2}
T= X^1 + \delta (X^1)^{-1}\; \; \text{and} \;\; T= X^2 + \delta (X^2)^{-1}.
\ee
\end{rem} 

\begin{lem}\label{ell.lem.20} Let $\de\in(0,1)$, let $T\in \b(\h)$ and let $\sigma(T) \subseteq K_\delta$. 
Then $T$ is generic  \forkde      ~if and only if there exists a generic prepair for $T$.
\end{lem}

\begin{proof} If $T$ is generic \forkde ~then the construction described in the paragraph following Definition \ref{ell.def.20} shows that there is a generic prepair for $T$.  Conversely, suppose that 
$T\in \b(\h)$ and  $\sigma(T) \subseteq K_\delta$, and suppose that there exists a generic prepair $X=(X^1,X^2)$  for $T$, where $X^1,X^2 \in \b(\h)$. 
By property (iii) of Definition \ref{ell.def.30}, $X^1X^2=\de$, and so $X^1$ and $X^2$ are both nonsingular since $\dim \h =n < \infty$. Therefore $0$ does not lie in either $\sigma(X^1)$ or $\sigma(X^2)$.
 By assumption $X=(X^1,X^2)$ is a generic prepair for $T$, and so, for $j=1,2$,  $X^j$ has $n$ distinct eigenvalues $\lam^j_1,\dots,\lam^j_n$. Let corresponding eigenvectors of $X^1$ be $e_1, \dots,e_n$.  
Since also $X^1X^2 = \de$, we have $X^2=\de(X^1)\inv$, and so (possibly after  renumbering) $\lam^2_j=\de(\lam^1_j)\inv$ for $j=1,\dots,n$, and an eigenvector of $X^2$ corresponding to $\lam^2_j$ is the eigenvector $e_j$ of $X^1$ corresponding to $\lam^1_j$.  Further, $T=X^1+X^2=X^1+\de (X^1)\inv$, and so $e_j$ is an eigenvector of $T$ corresponding to the eigenvalue $\lam^1_j + \de(\lam^1_j)\inv$.
\end{proof}

We denote by $\P$ the set of polynomials over $\c$.  Note that if $T \in \b(\h)$ has a generic prepair then $T$ has exactly $2^n$ generic prepairs.
If $T$ is generic \forkde, then for each choice of generic prepair $X=(X^1,X^2)$ one can define a pair  of \emph{residues} $A=(A^1,A^2)$ to be the unique pair of operators in the algebra generated by $T$ satisfying
\be\label{ell.40}
\frac{1-\frac12 zT}{1-zT +\delta z^2}=\frac{A^1}{1-zX^1}+\frac{A^2}{1-zX^2} \quad \mbox{ for all } z \in \d.
\ee
It is easy to verify that
\[
A^1=\frac12\ \  \text{ and }\ \  A^2 = \frac12
\]
for all choices of generic prepairs. The following lemma gives a valuable formula which relates prepairs and residues to strange dilations.
\begin{lem}\label{ell.lem.40} Let $\de\in (0,1)$.
Suppose that $T\in \b(\h)$ is  generic \forkde and $X$ is a generic prepair for $T$ with residues $A$. If $(\k, I, U)$ is a strange dilation of $T$, then
\[
A^1p(X^1) + A^2p(X^2) =I^*p(U)I
\]
for all $p \in \P$.
\end{lem}
\begin{proof} 
Under the assumptions of the lemma, since $(\k, I, U)$ is a strange dilation of $T$, then the formula
\be\label{ell.30.sd1}
\frac{1-\tfrac12 z T}{1-zT +\delta z^2}= I^*\frac{1}{1-zU} I
\ee
 holds for all $z\in \d$. The  formula  \eqref{ell.40} implies that, for  all $z\in \d$, 
\[
\frac{\half}{1-zX^1}+\frac{\half}{1-zX^2}=
 I^*\frac{1}{1-zU} I.
\]
If we expand the analytic functions in this formula into power series and equate coefficients we obtain the relations
\be\label{coeff}
  (X^1)^k +(X^2)^k =2 I^* U^k I, \;\;\;  k = 1,2, \dots.
\ee
The lemma follows if one takes linear combinations of  equations of the form \eqref{coeff}.
\end{proof}

\section{Cutting down strange dilations}\label{cutting}

We have shown in Proposition \ref{str-dil-inf} that if $T$ is the operator of multiplication by a scalar $\mu\in G_\de$
 on a finite- or infinite-dimensional Hilbert space $\h$, then necessarily, for any  strange dilation $(\k,I,U)$ of $T$, $\k$ is infinite dimensional. 
 One can obtain a {\em finite}-dimensional dilation of $T$ if one is 
 willing to replace the unitary $U$ in Definition \ref{ell.def.10} with a contraction $Y$.
 This idea is formalized in the following definition.
\begin{defin}\label{ell.def.40}
Let $\de\in (0,1)$ and let $T\in \b(\h)$ where $\h$ is a finite-dimensional Hilbert space and assume that $\sigma(T) \subseteq K_\delta$. We say that a triple $(\l,E,Y)$ is \emph{an even  stranger dilation of $T$ relative to $K_\de$ } if $\l$ is a Hilbert space with $\dim \l = 2\dim \h$, $E:\h \to \l$ is an isometry, $Y\in \b(\l)$ is a contraction,
and
\be\label{ell.50}
\frac{1-\frac{1}{2}  zT}{1-zT +\delta z^2}=E^*\frac{1}{1-zY}E \quad \text{ for all }z\in\d.
\ee
\end{defin}
  We may omit the phrase  ``relative to $K_\de$"  when it is clear from the context.
\begin{prop}\label{ell.prop.30}
Let $\de \in (0,1)$, let $T\in \b(\h)$ and let $\dim \h$ be finite. Assume that $T$ is generic for $K_\de$. Then $T$ has a strange dilation  relative to $K_\de$  if and only if $T$ has an even stranger dilation  relative to $K_\de$.
\end{prop}
\begin{proof}
First assume that $T$ has a strange dilation $(\k, I, U)$ relative to $K_\de$, as in Definition \ref{ell.def.10}. Fix a generic prepair $X=(X^1,X^2)$ of $T$ with residue $A=(A^1,A^2)$ and assume that $e_i$ is an eigenvector of $T$ corresponding to the eigenvalue $\mu_i$, as in Definition \ref{ell.def.20} and
$\lambda_i^r$ is the corresponding eigenvalue of $X^r$, so that
\be\label{ell.33-2}
X^r e_i = \lambda_i^r e_i  \qquad \mbox{ for } r=1,2\;\;  \mbox{ and } \; i=1,\ldots,n.
\ee
 For $i=1,\ldots,n$, let
\[
\I_i=\set{p\in \P}{p(\lambda_i^r)=0, r=1,2}
\]
and define closed linear subspaces $\m_i$ and $\n_i$ in $\k$ by 
\[
\m_i = \overline{\text{span}\set{p(U)Ie_i}{  p \in \P}}
\]
and
\[
\n_i = \overline{\text{span}\set{p(U)Ie_i}{p \in \I_i}}.
\]
Note that $\m_i$ and $\n_i$ are invariant for $U$, $\n_i \subseteq \m_i$, and that $\n_i$ has codimension at most 2 in $\m_i$.

Let
\[
\m = \sum_{i=1}^n \m_i\qquad \text{ and }\qquad \n= \sum_{i=1}^n \n_i.
\]
Since for each $i$, both $\m_i$ and $\n_i$ are invariant for $U$, so also $\m$ and $\n$ are invariant for $U$. Consequently, if we define a Hilbert space $\l$ by
\[
\l = \m \ominus \n =\set{y\in \m}{y \perp \n},
\]
and $Y\in \b(\l)$ by
\[
Y=P_\l U |\l,
\]
then
\be\label{ell.60}
p(Y) = P_\l p(U)|\l
\ee
for all $p\in \P$. Also, since for each $i$, $\n_i$ has codimension at most 2 in $\m_i$,
\be\label{ell.70}
\dim \l \le 2n.
\ee

Now notice that if $x=\sum_i c_ie_i \in \h$, then
\[
Ix = \sum_{i} c_i Ie_i \in \sum_i  \m_i =\m.
\]
Therefore, $\ran I \subseteq \m$.  Also observe that if $p \in \I_i$ and $y=p(U)Ie_i$, then by Lemma \ref{ell.lem.40}
\[
2I^* y =2I^*p(U)Ie_i =(A^1p(X^1) + A^2p(X^2))e_i =p(\lambda_i^1)A^1e_i+p(\lambda_i^2)A^2e_i=0.
\]
Therefore, $\n_i \subseteq \ker I^*$ for each $i$. Consequently, $\n \subseteq \ker I^*$, or equivalently, $\ran I \subseteq \n^\perp$. As we have both  $\ran I \subseteq \m$ and $\ran I \subseteq \n^\perp$,
\be\label{ell.80}
\ran I \subseteq \l.
\ee

Now observe that if we define $E:\h \to \l$, by the formula
\[
E(x) = P_\l I(x)
\]
then inclusion \eqref{ell.80} implies that $E$ is an isometry. Also, using formula \eqref{ell.60}, we deduce that if $k\ge 0$, then
\be\label{ell.80.2}
I^*U^kI =  (I^*P_\l) (P_\l U^k|\l ) (P_\l I)=(I^*P_\l) (P_\l U|\l )^k (P_\l I)=E^*Y^kE.
\ee
Therefore, if $z \in \d$,
 \begin{align*}
  \frac{1-\frac12 zT}{1-zT +\delta z^2}
 =& I^*\frac{1}{1-zU} I
= \sum_{k=0}^\infty z^k I^*U^kI\\ \\
  =&\sum_{k=0}^\infty z^k E^*Y^kE
   = E^*\frac{1}{1-zY}E,
 \end{align*}
that is, the formula \eqref{ell.50} holds for all $z\in \d$.

 To summarize, we have shown that $E$ is an isometry and the formula \eqref{ell.50} holds for all $z\in \d$.  It is clear that  $Y=P_\l U|\l$ is a contraction, since $U$ is unitary. \\
 
 There remains to show that $\dim \l = 2\dim \h$. For each $i=1,\ldots,n$ and $r=1,2$, choose $\chi_i^r \in \P$ such that for each $i$ and $r$.
 \[
 \chi_i^r(\lambda_i^s) =
\left\{
	\begin{array}{ll}
		1  & \mbox{if } s=r \\
		0 & \mbox{if } s \not=r 
	\end{array}.
\right.
 \] 
Note that if $p\in \P$, then, by equation \eqref{ell.80.2} and by Lemma \ref {ell.lem.40},
\begin{align} \label{ell.80.3}
E^* \big(p(Y)\chi_i^r(Y)Ie_i\big) &= \big(A^1p(X^1)\chi_i^r(X^1) + A^2p(X^2)\chi_i^r(X^1)\big)e_i  \notag\\
&=\big(\delta^1_r p(\lambda_i^1)A^1 + \delta^2_r p(\lambda_i^2)A^2\big)e_i,
\end{align}
where $\delta^k_r$ is the Kronecker symbol.
 We claim that the set $S$ of vectors in $\l$ defined by
 \[
 S=\set{\chi_i^r(Y)Ie_i}{i=1,\ldots,n,\ \ r=1,2}
 \]
 is linearly independent. For if $c_i^r$ are complex numbers with
\[
\sum_{i,r}c_i^r\ \chi_i^r(Y)Ie_i =0,
\]
and $p\in \P$, then by equation \eqref{ell.80.3},
\begin{align*}
0&=E^* p(Y)\sum_{i,r}c_i^r\ \chi_i^r(Y)Ie_i\\
&=\sum_{i,r}c_i^r\  E^*p(Y)\chi_i^r(Y)Ie_i\\
&=\sum_{i,r}c_i^r\  \big(\delta^1_r p(\lambda_i^1)A^1 + \delta^2_r p(\lambda_i^2)A^2\big)e_i.
\end{align*} 
But since $p\in \P$ is arbitrary and the residues $A^1$ and $A^2$ are invertible, this calculation shows that the set $S$ is linearly independent. In particular, $\dim \l \ge 2n$. Hence, in light of inequality \eqref{ell.70}, $\dim \l =2n$, as was to be proved.

Now assume that $(\l,E,Y)$ is an even stranger dilation of $T$  relative to $K_\de$. As in particular, $Y$ is a contraction, by the Nagy Dilation Theorem, there exists a Hilbert Space $\k$, an isometry $F:\l \to \k$, and a unitary $U\in \b(\k)$ such that
\[
p(Y)=F^*p(U)F
\]
for all $p\in \P$. Consequently, if we set $I=FE$, then it follows that $(\k,I,U)$ is a strange dilation of $T$ relative to $K_\de$.
\end{proof}

\section{The structure of even stranger dilations}

In this section we characterize generic operators $T$ with $W(T) \subseteq K_\de$ using even stranger dilations of $T$.

\begin{lem}\label{ell.lem.50} Let $\de \in (0,1)$ and let $T\in \b(\h)$ be generic for $K_\de$. Assume that $X=(X^1,X^2)$ is a generic prepair for $T$. If $(\k,E,Y)$ is an even stranger dilation of $T$  relative to $K_\de$, then $Y$ is similar to $ X^1\oplus X^2$, that is, there exists an invertible operator $S\in \b(\h \oplus \h,\k)$
such that $ Y=S(X^1\oplus X^2)S^{-1}$.
\end{lem}
\begin{proof}
If $(\k,E,Y)$  is an even  stranger dilation of $T$  relative to $K_\de$, then equations \eqref{ell.40} and \eqref{ell.50} imply that
\be\label{ell.90}
\frac{\frac{1}{2}}{1-z X^1} + \frac{\frac{1}{2}}{1-zX^2} = E^*\frac{1}{1-zY} E
\quad \text{ for all }z\in\d.
\ee
Let $S_1=1/\sigma(X^1) \cup 1/\sigma(X^2)$ and $S_2=1/\sigma(Y)$.  Since $\dim \h <\infty$ and $\dim \k = 2\dim \h$, $S_1$ and $S_2$ are finite sets. Furthermore, the left hand side of equation \eqref{ell.90} is analytic on $\c \setminus S_1$ and the right hand side of equation \eqref{ell.90} is analytic on $\c \setminus S_2$.  Since $\frac{\frac{1}{2}}{1-z X^1} + \frac{\frac{1}{2}}{1-zX^2}$ and $E^*\frac{1}{1-zY} E$ are analytic functions on $\c \setminus (S_1 \cup S_2)$ that agree on a nonempty open set, they agree on $\c \setminus (S_1 \cup S_2)$.
Therefore, equation $\eqref{ell.90}$ holds for all $z\in \c \setminus (S_1 \cup S_2)$.

Fix $\lambda\in \sigma(X^1)$ with corresponding eigenvector $e$. Then since $\sigma(X^1)$ and $\sigma(X^2)$ are disjoint, $\norm{(z-X^1)^{-1}e}$ is unbounded and  $\norm{(z-X^2)^{-1}e}$ is bounded on $\set{z\in \c}{0<|z-\lambda|<\eps}$ for all 
sufficiently small $\eps>0$. Consequently, equation \eqref{ell.90} implies that $\norm{(z-Y)^{-1}e}$ is unbounded on $\set{z\in \c }{0<|z-\lambda|<\eps}$  
for all sufficiently small $\eps >0$. Therefore, $\lambda \in \sigma(Y)$.

The previous paragraph showed that $\sigma(X^1) \subseteq \sigma(Y)$. By a similar argument it follows that  $\sigma(X^2) \subseteq \sigma(Y)$ as well.  Since $X=(X^1, X^2)$ is a  generic prepair, $\sigma(X^1) \cup \sigma(X^2)$ consists of $2\dim \h$ points and $\dim \k = 2\dim \h$. Let $\dim \h =n$,
\[
\sigma(X^1)= \{\lambda_i: i=1,2, \dots, n \} \; \text{ and } \; 
\sigma(X^2)= \{\lambda_i: i=n+1,\;n+2, \dots, 2n\}.
\]
There are vectors
\[ 
u_1,  \dots, u_{2n} \in \h \oplus \h
\]
such that
\[ 
(X^1\oplus X^2) u_i = \lambda_i u_i \; \text{  for} \; i=1,2, \dots, 2n.
\]
Since $Y$ acts on a $2n$-dimensional space and $\sigma(X^1) \cup \sigma(X^2)\subseteq \sigma(Y)$, it follows that 
 $\sigma(Y)= \sigma(X^1) \cup \sigma(X^2)$. Hence  there are  vectors
\[ 
v_1,  \dots, v_{2n} \in \k
\]
such that 
\[
Y v_i = \lambda_i v_i \; \text{  for} \; i=1,2, \dots, 2n.
\]
The vectors $u_i, i= 1,\dots,2n,$ and $v_i, i= 1,\dots,2n,$ comprise bases for $\h \oplus \h$ and $\k$ respectively, since they are both $2n$ in number and correspond to distinct eigenvalues.
Define an operator
\[
S: \h \oplus \h \to \k\; \text{ by } \; S u_i = v_i\; \text{  for} \; i=1,2, \dots, 2n.
\]
Then 
\[
Y S u_i = Y v_i = \lambda_i v_i = \lambda_i S u_i,
\]
and so
\[
S^{-1} Y S u_i = S^{-1}(Y v_i )= S^{-1}(\lambda_i v_i) = \lambda_i u_i=(X^1\oplus X^2) u_i\; \text{  for} \; i=1,2, \dots, 2n.
\]
\color{black}
It follows that $S\in \b(\h \oplus \h,\k)$ is an invertible operator 
such that $ Y=S(X^1\oplus X^2)S^{-1}$.
\end{proof}

\begin{lem}\label{ell.lem.60} Let $\de \in (0,1)$ and let
 $T\in \b(\h)$ be generic for $K_\de$. Assume that $X=(X^1,X^2)$ is a generic prepair for $T$. Then $(\k,E,Y)$ is an even  stranger dilation of $T$  relative to $K_\de $  if and only if there exist an invertible $S\in \b(\h \oplus \h,\k)$ and $C\in \b(\h, \h\oplus \h)$ such that
\be\label{ell.100}
Y=S(X^1\oplus X^2)S^{-1}\ \ \text{ and }\ \ E=S C,
\ee
and where $S$ and $C$ satisfy the following three conditions.
\begin{enumerate}[(i)]
\item
\[
(X^1\oplus X^2)^* S^*S (X^1\oplus X^2) \le S^*S.
\]
\item If we decompose $C$ as
\[
C =\begin{bmatrix}C^1 \\ C^2\end{bmatrix},
\]
then, for each $i=1,2$,  $C^i$ is invertible and commutes with $X^i$ and $T$. 
\item If $C$ is decomposed as in Condition (ii) above, then
\[
S^*S \begin{bmatrix}C^1 \\ C^2\end{bmatrix} =\frac12 \begin{bmatrix}{(C^1)^*}^{-1} \\ {(C^2)^*}^{-1}\end{bmatrix}.
\]
\end{enumerate}
\end{lem}
\begin{proof}
First assume that $(\k,E,Y)$ is an even stranger dilation of $T$.
The  formula
\be\label{ell.30.sd2}
\frac{1-\tfrac12 z T}{1-zT +\delta z^2}= E^*\frac{1}{1-zY} E
\ee
 holds for all $z\in \d$. The  formula \eqref{ell.40} implies that, for  all $z\in \d$,
\[
\frac{\frac12}{1-zX^1}+\frac{\frac12}{1-zX^2}=
 E^*\frac{1}{1-zY} E.
\]
If we expand the analytic functions in this formula into power series and equate coefficients we obtain 
\[
  \frac12 (X^1)^k +\frac12 (X^2)^k = E^* Y^k E,   \quad k=1,2, \dots.
\]
Let us take linear combinations of  the last equations to get
\color{black} 
\be\label{ell.110}
\frac12 p(X^1) + \frac12 p(X^2) =E^*p(Y)E
\ee
for all $p\in \P$.

By Lemma \ref{ell.lem.50}, there exists an invertible $S\in \b(\h \oplus \h,\k)$ such that
\be\label{ell.105}
Y=S(X^1\oplus X^2)S^{-1}.
 \ee
 Fix such an $S$ and define $C\in \b(\h, \h\oplus \h)$ by $C=S^{-1}E$. Clearly,  equation \eqref{ell.100} holds. Also, as $\norm{Y} \le 1$, $Y^*Y \le 1$, which implies that Condition (i) holds.

To see that Condition (ii) holds, let $B \in \b(\h \oplus \h,\h)$ be defined by $B=E^*S$. If we decompose $B$ as
\[
B=\begin{bmatrix}B^1&B^2\end{bmatrix},
\]
then, by equations  \eqref{ell.110}, \eqref{ell.105} and the definitions of $B$ and $C$, it follows that for all $p \in \P$,
\begin{align}\label{ell.106}
\frac12 p(X^1) + \frac12 p(X^2) &=E^*p(Y)E \nonumber\\
&=BS^{-1} p(Y) SC  \nonumber\\
&=Bp(X^1 \oplus X^2) C  \nonumber\\
&=B^1p(X^1)C^1 +B^2p(X^2)C^2.
\end{align}
By assumption, $X$ is a generic prepair for $T$, and so   $\sigma(X^1)$ and $\sigma(X^2)$ are disjoint and finite. Hence there exists a polynomial $p_0$ such that $p_0(z) =1$ for all $z \in
\sigma(X^1)$ and $p_0(z) =0 \; $  for all $z \in \sigma(X^2)$. In equation \eqref{ell.106}
replace $p$ by $p p_0$ and observe that 
 $p_0(X^1) =1$ and $p_0(X^2) =0$, and so equation \eqref{ell.106} becomes
 \be\label{ell.120}
 \frac12 p(X^1) = B^1p(X^1)C^1
 \ee
 for all $p \in \P$.
 Similarly,
 \be\label{ell.121}
 \frac12 p(X^2) = B^2p(X^2)C^2
 \ee
 for all $p \in \P$.
 Let $p=1$ in equations \eqref{ell.120} and \eqref{ell.121}, to get 
\[
B^1C^1=\frac12\ \ \text{ and }\ \ B^2C^2=\frac12,
\]
which implies that $C^1$ and $C^2$ are invertible. 
 Let $p=z$ in equations \eqref{ell.120} and  \eqref{ell.121}, to get 
 \be\label{x1c1x2c2}
 \frac12 X^1 =\frac12 {C^1}^{-1}X^1 C^1\ \ \text{ and }\ \
  \frac12 X^2=\frac12 {C^2}^{-1}X^2 C^2,
 \ee
 which implies that $C^1$  commutes with $X^1$ and $C^2$ commutes with $X^2$. 
 By equation \eqref{TX1X2},
\[
T= X^1 + \delta (X^1)^{-1}\; \; \text{and} \;\; T= X^2 + \delta (X^2)^{-1}.
\]
Therefore, 
 $C^1$  and $C^2$ commute with $T$.
\color{black}

 Finally, to see that Condition (iii) holds, let $B$ be as in the previous paragraph and compute that
 \begin{align*}
 S^*S \begin{bmatrix}C^1 \\ C^2\end{bmatrix}&=S^* E\\
 &=(E^*S)^*\\
 &=\begin{bmatrix}B^1&B^2\end{bmatrix}^*\\
 &=\frac12 \begin{bmatrix}{(C^1)^*}^{-1} \\ {(C^2)^*}^{-1}\end{bmatrix}.
 \end{align*}

 Now assume that $S\in \b(\h \oplus \h,\k)$ is invertible, $C\in \b(\h, \h\oplus \h)$, and that $S$ and $C$ satisfy Conditions (i), (ii), and (iii). We wish to show that if $Y$ and $E$ are defined by formulae \eqref{ell.100}, then $(\k,E,Y)$ is an even  stranger dilation of $T$. First observe that since $S$ is invertible, $\dim \k =2 \dim \h$. Also note that as $Y=S(X^1\oplus X^2)S^{-1}$, Condition (i) implies that $Y^*Y \le 1$, so that $\norm{Y} \le 1$. There remains to show that $E$ is an isometry and that equation \eqref{ell.50} holds.

Using Condition (iii) we see that
\[
E^*E=C^*(S^*S C)= \begin{bmatrix}{(C^1)^*} & {(C^2)^*}\end{bmatrix}\Big(\frac12 \begin{bmatrix}{(C^1)^*}^{-1} \\ {(C^2)^*}^{-1}\end{bmatrix}\Big)
=1.
\]
Therefore $E$ is an isometry.

To prove that equation \eqref{ell.50} holds, fix $z\in \d$ and note that
\begin{align*}
E^*\frac{1}{1-zY}E &=(C^*S^*) \frac{1}{1-zS(X^1\oplus X^2)S^{-1}}(SC) &~\\
&=(C^*S^*) S\frac{1}{1-z(X^1\oplus X^2)}S^{-1}(SC) &~\\
&=(C^*S^*S)\frac{1}{1-z(X^1\oplus X^2)}C &~\\
&=\frac12 \begin{bmatrix}{{C^1}^{-1}} &{{C^2}^{-1}}\end{bmatrix}
\begin{bmatrix}\frac{1}{1-zX^1}&0\\
0&\frac{1}{1-zX^2}\end{bmatrix}
\begin{bmatrix}C^1 \\ C^2\end{bmatrix}
\qquad   & \text{by Condition (iii)}\\
&=\frac12\  {C^1}^{-1}\frac{1}{1-zX^1}C^1+
\frac12\  {C^2}^{-1}\frac{1}{1-zX^2}C^2&~ \\
&=\frac{\frac12}{1-zX^1}+\frac{\frac12}{1-zX^2} \qquad  \qquad \qquad & \text{by Condition (ii)}\\
&=\frac{1-\frac12 zT}{1-zT +\delta z^2}
\qquad  \qquad \qquad \qquad  &  \text{by equation  \eqref{ell.40}}.
\end{align*}
 The proof of Lemma \ref{ell.lem.60} is complete. 
\end{proof}

\section{The existence of even stranger dilations}
\begin{prop}\label{ell.prop.40} Let $\de\in (0,1)$ and
let $T\in \b(\h)$ be generic for $K_\de$. Assume that $X=(X^1,X^2)$ is a generic prepair for $T$. There exists an even stranger dilation of $T$ relative to $K_\de$  if and only if there exists a self-adjoint operator $\Delta \in \b(\h \oplus \h)$ satisfying
\be\label{ell.150}
(X^1\oplus X^2)^*\Delta (X^1\oplus X^2) \le \Delta
\ee
and
\be\label{ell.160}
\Delta\begin{bmatrix}1\\1\end{bmatrix}=\frac12 \begin{bmatrix}1\\1\end{bmatrix}.
\ee
Furthermore, if $\Delta$ is a self-adjoint operator in $\b(\h \oplus \h)$ satisfying conditions \eqref{ell.150} and \eqref{ell.160}, then $\Delta$ is strictly positive definite.
\end{prop}
\begin{proof}
First assume that $(\k, E,Y)$ is an even stranger dilation for $T$. Let $S$ and $C$ be as in Lemma \ref{ell.lem.60} and define $\Delta$ by
\[
\Delta= (C^1\oplus C^2)^*S^*S(C^1\oplus C^2).
\]
We have 
\begin{align*}
(X^1\oplus X^2)^*\Delta (X^1\oplus X^2)
=&(X^1\oplus X^2)^*(C^1\oplus C^2)^*S^*S(C^1\oplus C^2) (X^1\oplus X^2)\\
=&
(C^1\oplus C^2)^*(X^1\oplus X^2)^*S^*S(X^1\oplus X^2) (C^1\oplus C^2)\\
&\qquad \text{by equation \eqref{x1c1x2c2}}\qquad\\
\le&
(C^1\oplus C^2)^*S^*S (C^1\oplus C^2)\\
&\qquad \text{by Condition (i) of Lemma \ref{ell.lem.60}}\qquad\\
&=\Delta,
\end{align*}
that is, inequality \eqref{ell.150} holds. Also,
\begin{align*}
\Delta\begin{bmatrix}1\\1\end{bmatrix}&=
(C^1\oplus C^2)^*S^*S
(C^1\oplus C^2)\begin{bmatrix}1\\1\end{bmatrix}\\
&=(C^1\oplus C^2)^*S^*S\begin{bmatrix}C^1 \\ C^2\end{bmatrix}\\
&=(C^1\oplus C^2)^*
\Big(\frac12 \begin{bmatrix}{(C^1)^*}^{-1} \\ {(C^2)^*}^{-1}\end{bmatrix}\Big)\\
&\qquad \text{by Condition (iii) of Lemma \ref{ell.lem.60} }\qquad\\
&=\frac12 \begin{bmatrix}1\\1\end{bmatrix},
\end{align*}
that is, equality \eqref{ell.160} holds.

Now assume that $\Delta$ is a self-adjoint operator in $\b(\h \oplus \h)$ satisfying the relations \eqref{ell.150} and \eqref{ell.160}. We first show that $\Delta$ is strictly positive definite and then show that there exists an even stranger dilation for $T$.

To prove that $\Delta>0$, observe that if we let
\be\label{defP}
P=\Delta-(X^1\oplus X^2)^*\Delta (X^1\oplus X^2),
\ee
then inequality \eqref{ell.150} guarantees that $P\ge0$.
Recall that $\sigma(X^i)\subset R_\de\subset \d$, and hence $\sigma(X^1\oplus X^2)\subseteq \d$.
Premultiply equation \eqref{defP} by ${(X^1\oplus X^2)^n}^*$ and postmultiply by
$(X^1\oplus X^2)^n$ to obtain
\[
{(X^1\oplus X^2)^n}^* P(X^1\oplus X^2)^n  = {(X^1\oplus X^2)^n}^* \Delta
(X^1\oplus X^2)^n - {(X^1\oplus X^2)^{n+1}}^*\Delta  (X^1\oplus X^2)^{n+1}
\]
Sum this telescoping series for $n$ from $0$ to $N$ to get
\[
\sum_{n=0}^N {(X^1\oplus X^2)^n}^* P(X^1\oplus X^2)^n  =  \Delta
 - {(X^1\oplus X^2)^{N+1}}^*\Delta  (X^1\oplus X^2)^{N+1}.
\]
Since $\sigma(X^1\oplus X^2) \subset \d, (X^1\oplus X^2)^{N+1} \to 0$ as $N \to \infty$, and therefore 
\[
\Delta = \sum_{n=0}^\infty {(X^1\oplus X^2)^n}^*P(X^1\oplus X^2)^n
\ge0.
\]
To see that $\ker \Delta =\{0\}$ assume to the contrary that $\ker \Delta \not=\{0\}$. Since $\Delta \ge 0$, the inequality \eqref{ell.150} implies that $\ker \Delta$ is invariant for $X^1\oplus X^2$. Therefore, as $\sigma(X^1) \cap \sigma(X^2) =\varnothing$, there exists $u\in \h$ such that $u\not=0$ and either $u\oplus 0\in \ker \Delta$ or $0 \oplus u \in \ker \Delta$. But if $u\oplus 0\in \ker \Delta$, equation \eqref{ell.160} implies that
\begin{align*}
\frac12 \ip{v}{u}_\h&= \ip{\frac12\begin{bmatrix}1\\1\end{bmatrix}v}
{\begin{bmatrix}u\\0\end{bmatrix}}_{\h \oplus \h}\\
&=\ip{\Delta\begin{bmatrix}1\\1\end{bmatrix}v}
{\begin{bmatrix}u\\0\end{bmatrix}}_{\h \oplus \h}\\
&=\ip{\begin{bmatrix}1\\1\end{bmatrix}v}
{\Delta \begin{bmatrix}u\\0\end{bmatrix}}_{\h \oplus \h}\\
&=0
\end{align*}
for all $v\in \h$, contradicting the fact that $u\not=0$. Similarly, if $0 \oplus u \in \ker \Delta$, we obtain a contradiction. This completes the proof that $\Delta>0$.

To construct an even stranger dilation for $T$, let $\k=\h \oplus \h$, define $E$ by
\[
E=\frac{1}{\sqrt2}\begin{bmatrix}1\\1\end{bmatrix},
\]
and let
\[
Y=\Delta^{\frac12}(X^1\oplus X^2)\Delta^{-\frac12}.
\]
Then $\dim \k =2\dim \h$ and the inequality \eqref{ell.150} implies that $\norm{Y}\le 1$. Also, as $E^*E=1$, $E$ is an isometry. Finally, noting that equation \eqref{ell.160} implies that
\[
\Delta^{\frac12}\begin{bmatrix}1\\1\end{bmatrix}=\frac{1}{\sqrt2} \begin{bmatrix}1\\1\end{bmatrix},
\]
we have, for $z \in \d$,
\begin{align*}
E^*\frac{1}{1-zY}E &=\Big(\frac{1}{\sqrt2} \begin{bmatrix}1\\1\end{bmatrix}\Big)^* \Big(\frac{1}{1-z\Delta^{\frac12}(X^1\oplus X^2)\Delta^{-\frac12}}\Big)\Big(\frac{1}{\sqrt2} \begin{bmatrix}1\\1\end{bmatrix}\Big)\\
&=\Big(\begin{bmatrix}1&1\end{bmatrix}\Delta^{\frac12}\Big)
\Big(\Delta^{\frac12}\frac{1}{1-z(X^1\oplus X^2)}\Delta^{-\frac12}\Big)
\Big(\Delta^{\frac12}\begin{bmatrix}1\\1\end{bmatrix}\Big)\\
&=\big(\begin{bmatrix}1&1\end{bmatrix}\Delta\big) \frac{1}{1-z(X^1\oplus X^2)}\begin{bmatrix}1\\1\end{bmatrix}\\
&=\frac12 \begin{bmatrix}1&1\end{bmatrix}
\frac{1}{1-z(X^1\oplus X^2)}\begin{bmatrix}1\\1\end{bmatrix}\\
&=\frac{\frac12}{1-zX^1}+\frac{\frac12}{1-zX^2}\\
&=\frac{1-\frac12 zT}{1-zT +\delta z^2},
\end{align*}
that is, formula \eqref{ell.50} holds.
\end{proof}

Observe that the above calculation is valid for {\em any} choice of a prepair $(X^1, X^2)$ for $T$, and so the existence statement in Proposition \ref{ell.prop.40} does not depend on the choice of a generic prepair for $T$.
\color{black}

The following proposition, which is essentially equivalent to Proposition \ref{ell.prop.40}, gives necessary and sufficient conditions for the existence of an even stranger dilation for $T$ directly in terms of $T$ rather than in terms of a prepair $X$ for $T$. If $T\in \b(\h)$ is generic for $K_\de$, then we may define $Z_T\in \b(\h \oplus \h)$ by
\be\label{ell.170}
Z_T=\frac12\ \begin{bmatrix} T&Q\\
Q& T\end{bmatrix},
\ee
where $Q$ is any operator in $\b(\h)$ commuting with $T$ and satisfying \be\label{ell.172}
Q^2=T^2-4\delta.
\ee
Just as $T$ has $2^{\dim\h}$ prepairs, there are $2^{\dim\h}$ choices for $Q$ satisfying equation \eqref{ell.172}. Also, the notation `$Z_T$' for the operator defined by formula \eqref{ell.170} is slightly ambiguous in that it does not reflect the dependence of $Z_T$ on the choice of $Q$.
\begin{prop}\label{ell.prop.50}
Let $\de\in(0,1)$. If $T\in \b(\h)$ is generic for $K_\de$, then the following conditions are equivalent.
\begin{enumerate}[(i)]
\item There exists an even  stranger dilation of $T$  relative to $K_\de$.
 \item For some (equivalently for all) choice of $Q$ satisfying equation \eqref{ell.172}, there exists a strictly positive definite $\Gamma \in \b(\h)$ such that
\be\label{ell.180}
\bignorm{\begin{bmatrix}1&0\\
0& \Gamma\end{bmatrix}^{\frac12}Z_T\begin{bmatrix}1&0\\
0& \Gamma\end{bmatrix}^{-\frac12}} \le 1.
\ee
 \item For some (equivalently for all) choice of $Q$ commuting with $T$ and satisfying $Q^2=T^2-4\delta$, there exists an invertible $S \in \b(\h)$ such that
\be\label{ell.182}
\bignorm{\begin{bmatrix}1&0\\
0& S\end{bmatrix}Z_T
\begin{bmatrix}1&0\\
0& S\end{bmatrix}^{-1}} \le 1.
\ee
\end{enumerate}
\end{prop}
\begin{proof}

Suppose Condition (i) holds. By Proposition \ref{ell.prop.40}, for a generic \forkde  $T$
and a  generic prepair $X=(X^1,X^2)$  for $T$, there exists an even stranger dilation of $T$ if and only if there exists a self-adjoint operator $\Delta \in \b(\h \oplus \h)$ satisfying
\be\label{ell.150Z}
(X^1\oplus X^2)^*\Delta (X^1\oplus X^2) \le \Delta
\ee
and
\be\label{ell.160Z}
\Delta\begin{bmatrix}1\\1\end{bmatrix}=\frac12 \begin{bmatrix}1\\1\end{bmatrix}.
\ee
\color{black}
If we define a symmetry  $U\in \b(\h\oplus \h)$  by
\[
U=\frac{1}{\sqrt2}\begin{bmatrix}1&1\\ 1&-1\end{bmatrix},
\]
then for each choice of prepair $X= (X^1, X^2)$ for $T$,
\begin{align}\label{ell.184}
U(X^1\oplus X^2)U= &  \frac{1}{\sqrt2}\begin{bmatrix}1&1\\ 1&-1\end{bmatrix}
\begin{bmatrix}X^1 &0\\ 0& X^2\end{bmatrix}\frac{1}{\sqrt2}\begin{bmatrix}1&1\\ 1&-1\end{bmatrix} \nonumber\\
= & \frac{1}{2}\begin{bmatrix}X^1 +X^2 &X^1 -X^2\\ X^1 -X^2 & X^1 +X^2\end{bmatrix}
\nonumber\\
= & \frac{1}{2}\begin{bmatrix} T & Q\\ Q & T\end{bmatrix}
= Z_T,
\end{align}
where $Q = X^1 -X^2$ is an operator in $\b(\h)$. By the definition \ref{ell.def.30} of a prepair $X= (X^1, X^2)$ for $T$, the equations  $X^1X^2= X^2X^1$, 
$X^1 +X^2 =T$ and $X^1 X^2=\delta$ hold. Thus one can check that, for $Q = X^1 -X^2$, $QT=TQ$ and $Q^2=T^2-4\delta$.
\color{black}
Hence, the inequality \eqref{ell.150Z} is equivalent to
\[
Z_T^* \big(U\Delta U\big)Z_T \le U\Delta U.
\]
Also, equation  \eqref{ell.160Z} is equivalent to
\[
\big(U\Delta U\big) \begin{bmatrix}1\\0\end{bmatrix}=\frac12
 \begin{bmatrix}1\\0\end{bmatrix}.
\]
Therefore, if we assume that $\Delta$ satisfies the relations \eqref{ell.150Z} and \eqref{ell.160Z}, then inequality \eqref{ell.180} holds if we let
\[
\begin{bmatrix}1&0\\
0& \Gamma\end{bmatrix}=2U\Delta U.
\]
This proves that Condition (i) implies Condition (ii).

Suppose Condition (ii) is satisfied. That is, there is a strictly positive definite $\Gamma \in \b(\h)$ such that equations \eqref{ell.180} holds.  Since $\Gamma > 0$, $\sigma(\Gamma)$ is a compact subset of $(0,\infty)$, and so $0\notin \sigma(\Gamma)$. 
Thus  $\Gamma$ is invertible and Condition (iii) is satisfied too.
\color{black}

 There remains to prove that Condition (iii) implies Condition (i). 
Let Condition (iii) hold, so that,  for some choice of an operator $Q$ in $\b(\h)$ commuting with $T$ and satisfying $Q^2=T^2-4\delta$, there exists an invertible $S \in \b(\h)$ such that
\be\label{ell.1820}
\bignorm{\begin{bmatrix}1&0\\
0& S\end{bmatrix}Z_T
\begin{bmatrix}1&0\\
0& S\end{bmatrix}^{-1}} \le 1.
\ee
Define
\be\label{ell.delta11}
\Delta_{11} = \frac{1}{4}( 1 + S^*S)
\ee
and 
\be\label{ell.delta}
\Delta = \begin{bmatrix} \Delta_{11}  & \frac{1}{2} - \Delta_{11}\\
 \frac{1}{2} - \Delta_{11} & \Delta_{11} \end{bmatrix}.
\ee
Then 
\[
U \Delta U =  U \begin{bmatrix} \Delta_{11}  & \frac{1}{2} - \Delta_{11}\\
 \frac{1}{2} - \Delta_{11} & \Delta_{11} \end{bmatrix} U= 
 \frac12 \begin{bmatrix} 1  & 0\\
 0 & -1 + 4\Delta_{11} \end{bmatrix}. 
\] 
Since, by definition, 
\[ \Delta_{11} = \frac{1}{4}( 1 + S^*S), \]
we have
\be \label{Delta-S}
\Delta = \frac{1}{2} U \begin{bmatrix}1&0\\
0& S^*S\end{bmatrix}U,
\ee
One can see that, $\Delta$ defined by equations \ref{ell.delta} and \ref{ell.delta11}, is  a self-adjoint operator $\Delta \in \b(\h \oplus \h)$.

By assumption of Condition (iii),  $S \in \b(\h)$ is invertible and  such that
\be\label{ell.1820Z}
\bignorm{\begin{bmatrix}1&0\\
0& S\end{bmatrix}Z_T
\begin{bmatrix}1&0\\
0& S\end{bmatrix}^{-1}} \le 1.
\ee
Hence
\[
1 - \left(\begin{bmatrix}1&0\\
0& S\end{bmatrix}^{-1}\right)^* Z_T^*  \begin{bmatrix}1&0\\
0& S\end{bmatrix}^*\begin{bmatrix}1&0\\
0& S\end{bmatrix}Z_T
\begin{bmatrix}1&0\\
0& S\end{bmatrix}^{-1} \ge 0.
\]
In view of equation \eqref{Delta-S}, the above inequality is equivalent to
\[
1 - \left(\begin{bmatrix}1&0\\
0& S\end{bmatrix}^{-1}\right)^* \;Z_T^* \; 2 U \Delta U \; Z_T\;
\begin{bmatrix}1&0\\
0& S\end{bmatrix}^{-1} \ge 0.
\]
Therefore
\[
\begin{bmatrix}1&0\\
0& S\end{bmatrix}^*\begin{bmatrix}1&0\\
0& S\end{bmatrix} - 
\begin{bmatrix}1&0\\
0& S\end{bmatrix}^*
\left(\begin{bmatrix}1&0\\
0& S\end{bmatrix}^{-1}\right)^* 
Z_T^*  2 U \Delta U Z_T
\begin{bmatrix}1&0\\
0& S\end{bmatrix}^{-1} 
\begin{bmatrix}1&0\\
0& S\end{bmatrix} \ge 0,
\]
and so
\[
  U \Delta U - 
Z_T^*  \; U \Delta U \; Z_T \ge 0.
\]
We have shown at the beginning of the proof that the last inequality 
implies that  $\Delta$ satisfies
\be\label{ell.1500}
(X^1\oplus X^2)^*\Delta (X^1\oplus X^2) \le \Delta
\ee
and
\be\label{ell.1600}
\Delta\begin{bmatrix}1\\1\end{bmatrix}=\frac12 \begin{bmatrix}1\\1\end{bmatrix}.
\ee
By Proposition \ref{ell.prop.40}, the existence of such self-adjoint operator $\Delta$  satisfying 
inequality \eqref{ell.1500} and equation \eqref{ell.1600} implies the existence of an even stranger dilation of $T$.
\end{proof}

\section{A characterization of the operators $T$ with $W(T) \subseteq {K}_\delta$}\label{characterization}
Let $\de\in (0,1)$.
If $\h$ is a Hilbert space we let
\[
\w_\delta(\h)=\set{T\in\b(\h)}{W(T)\subseteq K_\delta}.
\]
Let $T\in \b(\h)$ be generic for $K_\de$. Recall that  $Z_T\in \b(\h \oplus \h)$ is defined by
\be\label{ell.170.par}
Z_T=\frac12\ \begin{bmatrix} T&Q\\
Q& T\end{bmatrix},
\ee
where $Q$ is any operator in $\b(\h)$ commuting with $T$ and satisfying \be\label{ell.172.par}
Q^2=T^2-4\delta.
\ee
\begin{thm}\label{par.thm.10}
Let $T\in\b(\h)$ be generic for $K_\de$. Then $T\in \w_\delta(\h)$ if and only if there exists an invertible $S \in \b(\h)$ such that
\be\label{par.10}
\bignorm{\begin{bmatrix}1&0\\
0& S\end{bmatrix}Z_T\begin{bmatrix}1&0\\
0& S\end{bmatrix}^{-1}} \le 1.
\ee

\end{thm}
\begin{proof}
Since $T$ is generic for $K_\de$, $\sigma(T) \subseteq K_\de$. By Theorem \ref{ell.prop.20}, $W(T)\subseteq K_\de$ if and only if there is a strange dilation of $T$.  By Proposition \ref{ell.prop.30}, for $T$  which is generic for $K_\de$, $T$ has a strange dilation if and only if $T$ has an even stranger dilation.
By Proposition \ref{ell.prop.50}, for $T$  which is generic for $K_\de$, $T$ has an even stranger dilation if and only if the following condition is satisfied:
for some (equivalently for all) choice of $Q$ commuting with $T$ and satisfying $Q^2=T^2-4\delta$, there exists an invertible $S \in \b(\h)$ such that
\be\label{ell.1820ZZ}
\bignorm{\begin{bmatrix}1&0\\
0& S\end{bmatrix}Z_T
\begin{bmatrix}1&0\\
0& S\end{bmatrix}^{-1}} \le 1.
\ee
The statement of Theorem \ref{par.thm.10} follows.   \color{black}
\end{proof}

This theorem suggests that there should be a parametrization of the operators $T\in \b(\h)$ for which $W(T) \subseteq K_\delta$ in terms of contractions $C$ that act on $\h \oplus \h$. In the remainder of this section we shall explicitly describe such a parametrization.

When $\h$ is a Hilbert space we let $J$ denote the symmetry in $\b(\h\oplus \h)$ defined by
\[
J=\begin{bmatrix}1&0\\
0& -1\end{bmatrix}.
\]
If $T\in\b(\h)$ is generic \forkde  and $W(T) \subseteq K_\delta$, then using Theorem \ref{par.thm.10} we may define $C\in \b(\h \oplus \h)$ by the formula
\be\label{par.20}
C=\begin{bmatrix}1&0\\
0& S\end{bmatrix}Z_T\begin{bmatrix}1&0\\
0& S\end{bmatrix}^{-1}.
\ee
Evidently, then $C$ is a \emph{contraction}, that is,
\be\label{par.30}
\norm{C} \le 1.
\ee
Also, let $\dim \h =n$ and $\sigma(T) = \{\mu_1,\ldots,\mu_n\}$. 
As we showed in Proposition \ref{ell.prop.50}, since $T$ is generic \forkde  and $W(T)\subseteq K_\de$, $T$ has an even stranger dilation, and hence the implication (i)$\implies$(ii) of Proposition \ref{ell.prop.50} shows that, for every choice $(X^1,X^2)$ of generic prepair for $T$,
\[
U(X^1\oplus X^2)U = Z_T,
\]
where
\[
U=\frac{1}{\sqrt{2}} \bbm 1 & 1\\ 1 & -1 \ebm
\]
(see equation \eqref{ell.184}).  Note that $U^* =U$, so that $Z_T$ is unitarily equivalent to $X^1\oplus X^2$, and therefore
\begin{align*}
\sigma(C) &=\sigma(Z_T) = \sigma(X^1 \oplus X^2) = \sigma(X^1) \cup \sigma( X^2) \\
   &=\left\{ \frac{\mu_1}{2}\pm \frac{\sqrt{(\mu_1)^2-4\delta}}{2},\ldots,
\frac{\mu_n}{2}\pm \frac{\sqrt{(\mu_n)^2-4\delta}}{2} \right\}.
\end{align*}
As a consequence, as $T$ is generic for $K_\de$, $C$ is \emph{generic}, that is,
\be\label{par.40}
C \text{ has } 2n \text{ distinct eigenvalues.}
\ee
Another, more subtle, property of $C$ is that $C$ is \emph{J-equivalent} to $\delta C^{-1}$
\footnote{We say that two operators $A,B\in\b(\h\oplus \h)$ are $J$-equivalent if $A=JBJ$.},
 that is,
\be\label{par.50}
\frac{\delta}{C} =JCJ.
\ee
To prove equation \eqref{par.50} first observe that
 \begin{align*}
 JZ_TJ&=\begin{bmatrix}1&0\\0& -1\end{bmatrix}
 \frac12\begin{bmatrix} T&\sqrt{T^2-4\delta}\\
\sqrt{T^2-4\delta}& T\end{bmatrix}
\begin{bmatrix}1&0\\0& -1\end{bmatrix}\\ \\
&= \frac12\begin{bmatrix} T&-\sqrt{T^2-4\delta}\\
-\sqrt{T^2-4\delta}& T\end{bmatrix}\\ \\
&=\frac{\delta}{Z_T}.
 \end{align*}
 The last equation follows from the facts that $QT=TQ$ and $Q^2=T^2-4\de$.
 \color{black}
 Therefore,
  \begin{align*}
 JCJ&=J\big(
 \begin{bmatrix}1&0\\
0& S\end{bmatrix}Z_T\begin{bmatrix}1&0\\
0& S\end{bmatrix}^{-1}\big)J\\ \\
&=\begin{bmatrix}1&0\\
0& S\end{bmatrix}\big(JZ_TJ\big)\begin{bmatrix}1&0\\
0& S\end{bmatrix}^{-1}\\ \\
&=\begin{bmatrix}1&0\\
0& S\end{bmatrix}\big(\frac{\delta}{Z_T})\begin{bmatrix}1&0\\
0& S\end{bmatrix}^{-1}\\ \\
&=\frac{\delta}{C}.
 \end{align*}

A final property of $C$ that we wish to observe is that
\be\label{par.55}
0 \not\in \sigma(C_{12})
\ee
where $C_{12}$ is the 1-2 entry of $C$ when $C$ is represented as a $2\times 2$ block matrix acting on $\h \oplus \h$. To see the relation \eqref{par.55} note that equation \eqref{par.20} implies that $C_{12}=QS^{-1}$ and that if $Q$ is not invertible, then $4\delta \in \sigma(T^2)$, contradicting the genericity of $T$.

In light of the above properties of $C$ we make the following definition.
\begin{defin}\label{par.def.10}
Let $\de\in(0,1)$. If $\h$ is a finite-dimensional Hilbert space and $\dim \h=n$, we say that $C$ is a \emph{$(\delta,\h)$-germinator} if $C\in \b(\h\oplus \h)$ and $C$ satisfies the conditions:
\be\label{par.30.def}
\norm{C} \le 1,
\ee
\be\label{par.40.def}
C \text{ has } 2n \text{ distinct eigenvalues,}
\ee
\be\label{par.50.def}
\frac{\delta}{C} =JCJ,
\ee
and
\be\label{par.55.def}
0 \not\in \sigma(C_{12})
\ee
where $C_{12}$ is the 1-2 entry of $C$ when $C$ is represented as a $2\times 2$ block matrix acting on $\h \oplus \h$.
\end{defin}
\noindent Armed with this definition, we summarize  the above observations  in the following proposition.
\begin{prop}\label{par.prop.10}
Assume that $T\in \w_\delta (\h)$ is generic for $K_\de$. There exists an invertible $S\in \b(\h)$ such that if $C$ is defined by the equation 
\be\label{par.20.prop}
C=\begin{bmatrix}1&0\\
0& S\end{bmatrix}Z_T\begin{bmatrix}1&0\\
0& S\end{bmatrix}^{-1},
\ee
then $C$ is a $(\delta,\h)$-germinator.
\end{prop}
A converse to Proposition \ref{par.prop.10} hinges on the inversion of equation \eqref{par.20} which we turn to now.

\begin{lem}\label{par.lem.10}
If $T\in \b(\h)$  is generic for $K_\de$, then
\[
Z_T + \frac{\delta}{Z_T} =T \oplus T.
\]
\end{lem}
\begin{proof}
By equation \eqref{par.50}, $Z_T$ is \emph{J-equivalent} to $\delta Z_T^{-1}$.
Thus
\begin{align}\label{JZJ-T}
Z_T + \frac{\delta}{Z_T} &=Z_T + J Z_T J \\
&= \frac12\begin{bmatrix} T& \sqrt{T^2-4\delta}\\
\sqrt{T^2-4\delta}& T\end{bmatrix}+
\frac12\begin{bmatrix} T&-\sqrt{T^2-4\delta}\\
-\sqrt{T^2-4\delta}& T\end{bmatrix} \nonumber\\
&= T \oplus T.\nonumber
\end{align}
\end{proof}

\begin{lem}\label{par.lem.20}
If $C$ is a $(\delta,\h)$-germinator, then
\[
C + \frac{\delta}{C} =C_{1} \oplus C_{2},
\]
where $C_{2}$ is similar to $C_{1}$.
\end{lem}
\begin{proof}
Since $C$ is a $(\delta,\h)$-germinator, $\frac{\delta}{C} = J C J$, and so 
\begin{align}\label{JZJ}
C + \frac{\delta}{C} &=C + J C J \\
&= \begin{bmatrix} C_{11}& C_{12}\\
 C_{21} &  C_{22}\end{bmatrix}+
 \begin{bmatrix} C_{11}& -C_{12}\\
 -C_{21} &  C_{22}\end{bmatrix}\nonumber\\
&=  2C_{11} \oplus  2C_{22}.\nonumber
\end{align}
Accordingly, define $C_{1} = \frac12 C_{11}$ and $C_{2} = \frac12 C_{22}$.

Noting that
\[
 \begin{bmatrix} \delta& 0\\
 0 &  \delta\end{bmatrix} = (\frac{\delta}{C})\ C=\begin{bmatrix} C_{11}& -C_{12}\\
 -C_{21} &  C_{22}\end{bmatrix}
\begin{bmatrix} C_{11}& C_{12}\\
 C_{21} &  C_{22}\end{bmatrix}=
 \begin{bmatrix} C_{11}^2 - C_{12}C_{21} & C_{11}C_{12} - C_{12}C_{22}\\
 C_{22}C_{21} - C_{21}C_{11}&   C_{22}^2 - C_{21}C_{12}\end{bmatrix},
\]
we see from the 1-2 entry that
\[
 C_{11}C_{12} - C_{12}C_{22}=0.
\]
Therefore, the relation \eqref{par.55.def}
$$0 \not\in \sigma(C_{12}),$$
implies that $C_{12}$ is invertible and 
\[
 C_1=\half C_{11}=\half  C_{12}C_{22}C_{12}^{-1} = C_{12}C_2C_{12}\inv.
\]
Thus $C_{2}$ is similar to $C_{1}$.
\end{proof}

\begin{prop}\label{par.prop.20} Let $\h$ be a finite-dimensional Hilbert space and $\dim \h=n$, and let $C\in \b(\h\oplus \h)$ 
be a $(\delta,\h)$-germinator, then there exists an invertible $S\in \b(\h)$ and a  $T\in \w_\de(\h)$ such that $T$ is generic for $K_\de$ and the equation 
\be\label{par.20.prop.20}
C=\begin{bmatrix}1&0\\
0& S\end{bmatrix}Z_T\begin{bmatrix}1&0\\
0& S\end{bmatrix}^{-1}
\ee
holds.
\end{prop}
\begin{proof}
First observe using equation \eqref{par.50.def} that if
\[
C\begin{pmatrix}u\\v\end{pmatrix}=\lambda\begin{pmatrix}u\\v\end{pmatrix},
\]
then
\[
C\begin{pmatrix}u\\-v\end{pmatrix}=\frac{\delta}{\lambda}\begin{pmatrix}u\\-v\end{pmatrix}.
\]
Thus, using the fact \eqref{par.40.def} that $C$ has $2n$ distinct eigenvalues,
 we may decompose the set of eigenvalues of $C$ into a disjoint union
\[
\bigcup_{i=1}^n\  \{\lambda_i,\frac{\delta}{\lambda_i}\}.
\]

Since all eigenvalues of $C$ lie in $R_\de$, $\sigma(C)\subseteq R_\de$.
\color{black}
Furthermore, if $\begin{pmatrix}u_i\\v_i\end{pmatrix}$ and $\begin{pmatrix}u_i\\-v_i\end{pmatrix}$ are the eigenvectors corresponding to $\lambda_i$ and $\delta/\lambda_i$ respectively, then the statement  \eqref{par.40.def}, that $C$ has $2n$ distinct eigenvalues, implies that $\{u_1,u_2,\ldots,u_n\}$ is a basis for $\h$. Therefore, there exist $S,T\in \b(\h)$ satisfying
\be\label{par.56}
S u_i = v_i\qquad \text{and}\qquad T u_i=(\lambda_i+\frac{\delta}{\lambda_i})u_i,\qquad \ \ \ \ i=1,2,\ldots,n,
\ee
and the operator $Q$ defined by
\[
Q u_i =(\lambda_i-\frac{\delta}{\lambda_i})u_i,\qquad \ \ \ \ i=1,2,\ldots,n
\]
is a square root of $T^2-4\delta$ that commutes with $T$.
Observe that $T$ is the restriction of $C+\frac{\de}{C}$ to $\h \oplus \{0\}$ which is the span of $u_1,\dots,u_n$, 
and so
\[
\sigma(T)\subseteq \sigma(\pi(C)) \subseteq \pi(R_\de) \subseteq K_\de.
\]
By the definition of $T$, the eigenvalues of $T$ are the points $\mu_i=\pi(\lambda_i), i=1,\dots,n$.
Since $C$ is a $(\de,\h)$-germinator, $C$ has $2n$ distinct eigenvalues $\lambda_i, \frac{\de}{\lambda_i}, i=1,\dots,n$.  Therefore the equation $\pi(\lambda)=\mu_i$ has two distinct solutions in $R_\de$ for each $i$.  Hence $T$ is generic for $K_\de$.
\color{black}

To see that the equation \eqref{par.20.prop.20} holds it suffices to show that the two operators
\[
\begin{bmatrix}1&0\\
0& S\end{bmatrix}^{-1}C\qquad \text{ and }\qquad Z_T\begin{bmatrix}1&0\\
0& S\end{bmatrix}^{-1}
\]
agree on the eigenvectors of $C$ (which the relation \eqref{par.40.def} guarantees to be a basis for $\h \oplus \h$). But for fixed $i$,
\begin{align*}
\begin{bmatrix}1&0\\
0& S\end{bmatrix}^{-1}C\begin{pmatrix}u_i\\v_i\end{pmatrix}&=
\lambda_i \begin{bmatrix}1&0\\
0& S\end{bmatrix}^{-1} \begin{pmatrix}u_i\\v_i\end{pmatrix}=\lambda_i \begin{pmatrix}u_i\\u_i\end{pmatrix}\\ \\
&=\frac12 \begin{bmatrix}T&Q\\
Q& T\end{bmatrix}\begin{pmatrix}u_i\\u_i\end{pmatrix}=Z_T\begin{pmatrix}u_i\\u_i\end{pmatrix}\\ \\
&=Z_T\begin{bmatrix}1&0\\
0& S\end{bmatrix}^{-1}\begin{pmatrix}u_i\\v_i\end{pmatrix},
\end{align*}
and likewise,
\begin{align*}
\begin{bmatrix}1&0\\
0& S\end{bmatrix}^{-1}C\begin{pmatrix}u_i\\-v_i\end{pmatrix}&=
\frac{\delta}{\lambda_i} \begin{bmatrix}1&0\\
0& S\end{bmatrix}^{-1} \begin{pmatrix}u_i\\-v_i\end{pmatrix}
=\frac{\delta}{\lambda_i} \begin{pmatrix}u_i\\-u_i\end{pmatrix}\\ \\
&=\frac12 \begin{bmatrix}T&Q\\
Q& T\end{bmatrix}\begin{pmatrix}u_i\\-u_i\end{pmatrix}
=Z_T\begin{pmatrix}u_i\\-u_i\end{pmatrix}\\ \\
&=Z_T\begin{bmatrix}1&0\\
0& S\end{bmatrix}^{-1}\begin{pmatrix}u_i\\-v_i\end{pmatrix}.
\end{align*}
Therefore,
\[
\begin{bmatrix}1&0\\
0& S\end{bmatrix}^{-1} C = Z_T \begin{bmatrix}1&0\\
0& S\end{bmatrix}^{-1}
\]
and equation \eqref{par.20.prop.20} holds.  By Theorem \ref{par.thm.10}, $W(T)\subseteq K_\de$.
\end{proof}
Evidently, Theorem \ref{par.thm.10} may be restated in the following way.
\begin{thm}\label{thm.10}
Let $\de \in (0,1)$ and let $T\in \b(\h)$ be generic for $K_\delta$. Then $T \in \w_\delta(\h)$ if and only if
there exists a $(\delta,\h)$-germinator  $C\in \b(\h\oplus \h)$  and an invertible operator $S\in\b(\h)$ such that
\be\label{par.70}
C=\begin{bmatrix}1&0\\
0& S\end{bmatrix}Z_T\begin{bmatrix}1&0\\
0& S\end{bmatrix}^{-1}.
\ee
Moreover,  for any such $(\de,\h)$-germinator $C$,
\be \label{T-S}
T = \left(C+ \frac{\de}{C} \right)\big|_{\h\oplus \{0\}}.
\ee
\end{thm}
\begin{proof}
By Proposition \ref{par.prop.10}, if $T\in \w_\delta (\h)$ is generic for $K_\delta$, then  exists an invertible $S\in \b(\h)$ such that if $C$ is defined by the equation 
\be\label{par.20.prop.Z}
C=\begin{bmatrix}1&0\\
0& S\end{bmatrix}Z_T\begin{bmatrix}1&0\\
0& S\end{bmatrix}^{-1},
\ee
then $C$ is a $(\delta,\h)$-germinator.

Conversely, suppose that  $C\in \b(\h\oplus \h)$ is a $(\delta,\h)$-germinator and there is   an invertible operator $S\in\b(\h)$ such that
\be\label{par.70.2}
C=\begin{bmatrix}1&0\\
0& S\end{bmatrix}Z_T\begin{bmatrix}1&0\\
0& S\end{bmatrix}^{-1}.
\ee
By the definition of a $(\delta,\h)$-germinator, $ \| C\| \le 1$, and so equation 
\be\label{par.10.2}
 \bignorm{\begin{bmatrix}1&0\\
0& S\end{bmatrix}Z_T\begin{bmatrix}1&0\\
0& S\end{bmatrix}^{-1}}= \|C\| \le 1.
\ee
is satisfied. Thus, by Theorem \ref{par.thm.10}, $T\in \w_\delta(\h)$.\\

Note that, by equation \eqref{par.70},
\[
C =\begin{bmatrix}1&0\\
0& S\end{bmatrix}Z_T\begin{bmatrix}1&0\\
0& S\end{bmatrix}^{-1}= \frac12 \begin{bmatrix} T& Q S^{-1}\\
 S Q&  S T S^{-1}\end{bmatrix}.
\]
Thus, by Lemma \ref{par.lem.20}, since $C\in \b(\h\oplus \h)$ is a $(\delta,\h)$-germinator,
\[
C + \frac{\delta}{C} =2C_{11} \oplus 2 C_{22}=\begin{bmatrix} T& 0\\
 0&  S T S^{-1}\end{bmatrix}.
\]
Therefore,
\[
T = \left(C+ \frac{\de}{C} \right)\big|_{\h\oplus \{0\}}.
\]
\end{proof}

\section{A parametrization of the operators $T$ with $W(T) \subseteq {K}_\delta$} \label{parametrization}

In view of Theorem \ref{thm.10} it is relatively straightforward to parametrize 
$\w_\delta(\h)$.
\begin{lem}\label{form.lem.10} 
Let $\de\in (0,1)$, let $\h$ be a Hilbert space of finite dimension $n$ and let $C\in \b(\h \oplus \h)$.  Suppose that the operator $C$ satisfies the relations 
\be \label{par.3050}
\| C \| \leq 1   \text{ and } \frac{\de}{C} = JCJ.
\ee
 Then $C=\sqrt \delta XJ$ for some $X$ having the block matrix representation
\be\label{form.10}
X=\begin{bmatrix}1&E\\0&-1\end{bmatrix},
\ee
with respect to the decomposition
\be\label{form.20}
\h\oplus \h = \ker (1-X) \oplus \big(\ker (1-X)\big)^\perp,
\ee
and where the operator $E: \big(\ker (1-X)\big)^\perp \to \ker (1-X)$ satisfies
\be\label{form.30}
\norm{E} \le \frac{1}{\sqrt \delta}-\sqrt \delta.
\ee
\end{lem}
\begin{proof}
 Suppose that $C$ satisfies the relations \eqref{par.3050}. Then $\de=CJCJ$.  Let $X=\de^{-\half}CJ$.  Then $X^2=1$, so that, for any $u\in\h\oplus\h, (1+X)(1-X)u=(1-X)(1+X)u=(1-X^2)u=0$.  It follows that
\[
\ran (1+X) \subseteq \ker (1-X), \quad \ran (1-X) \subseteq \ker (1+X).
\]
Since
\[
x=\half(1+X)x +\half(1-X)x \quad \text{ for all } x\in \h\oplus\h,
\]
it follows that $\ker(1-X) + \ker(1+X)=\h\oplus\h$.  Since clearly $\ker(1-X) \cap \ker(1+X)=\{0\}$,
$\h\oplus \h$ is the vector space direct sum of $\ker(1-X)$ and $ \ker(1+X)$.  \color{black}
Let us write 
\[
V_-=\ker(1-X), \quad V_+= \ker(1+X)
\]
for the eigenspaces of $X$ corresponding to the eigenvalues $1, -1$ respectively.  The operator matrix of $X$ relative to the orthogonal decomposition $\h\oplus\h=V_- \oplus V_-^\perp$ has the form
\[
X \sim \bpm 1 & E \\ 0 & X_{22} \epm
\]
for some operators $E:V_-^\perp \to V_-$ and $X_{22}:V_-^\perp \to V_-^\perp$.
Since $X^2=1$, we have also $X_{22}^2 =1$ and therefore $\sigma(X_{22})\subseteq \{1,-1\}$, by the spectral mapping theorem.  However, $V_-^\perp$ does not contain any eigenvector of $X$ corresponding to the eigenvalue $1$, and so $\sigma(X_{22})= \{-1\}$.  Thus $1-X_{22}$ is invertible,
and so 
\[
0=\big(1-X_{22}\big)\inv \big(1-X_{22}^2\big)= 1+ X_{22},
\]
that is, $X_{22}=-1$.
The expression of $X$ by an operator matrix with respect to the direct decomposition $\h\oplus\h =V_-\oplus V_-^\perp$ becomes
\[
X \sim \bpm 1 & E \\ 0 & -1 \epm.
\]
Since $\|C\|\leq 1$, $\|X\| = \|\de^{-\half}CJ\|\leq \de^{-\half}$, and so $\de X^*X \leq 1$.  Thus
\be \label{deineq}
0 \leq 1-\de X^*X = \bbm 1-\de & -\de E \\ -\de E^* & 1-\de -\de E^*E \ebm
\ee
 For general matrices $A,B,D$ of appropriate sizes the identity
\[
\bbm 1 & 0\\-B^*A\inv & 1\ebm \bbm A & B \\ B^* & D \ebm \bbm 1 & -A\inv B \\ 0 & 1\ebm =\bbm A & 0\\ 0 & D-B^*A\inv B \ebm,
\]
is valid whenever the operator $A$ is invertible. Select $A=1-\de, B=-\de E, D=1-\de - \de E^*E$ to deduce that
\[
\bbm 1-\de & -\de E \\ -\de E^* & 1-\de -\de E^*E \ebm \mbox{ is congruent to } \bbm 1-\de & 0 \\ 0 & 1-\de - \frac{\de}{1-\de}E^*E \ebm.
\]
It follows from the inequality \eqref{deineq} that $1-\de - \frac{\de}{1-\de}E^*E \geq 0$, so that 
\[
E^*E \leq \frac{(1-\de)^2}{\de} = \big(\frac{1}{\sqrt{\de}} -\sqrt{\de}\big)^2.
\]
Hence $\|E\| \leq  \frac{1}{\sqrt{\de}} -\sqrt{\de}$.
\end{proof}

\begin{lem}\label{form.lem.15} Let $\de \in (0,1)$.
If $C$ is a $(\delta,\h)$-germinator, then there exist $A,F
\in \b(\h)$ such that $A$ is a self-adjoint contraction,
\be\label{form.35}
\norm{F}\le \frac{1}{\sqrt \delta}-\sqrt \delta,
\ee
 and
\be\label{form.40}
C_{11}=\sqrt\delta\ \big(A +\half \sqrt{1+A}\ F\ \sqrt{1-A}\big)
\ee
\end{lem}
\begin{proof}
Using Lemma \ref{form.lem.10} we see that
\[
C_{11}=\sqrt\delta\ V^*(XJ)V
\]
where $X$ is as in the relations \eqref{form.10}, \eqref{form.30}, and
\[
V:\h \to \ker (1-X) \oplus \big(\ker (1-X)\big)^\perp
\]
is an isometry. Decomposing $V$ into a block $2\times 1$ matrix
\[
V=\begin{bmatrix}V_1\\V_2\end{bmatrix}:\h \to \ker (1-X) \oplus \big(\ker (1-X)\big)^\perp
\]
we obtain the formula
\be\label{form.50}
C_{11}=\sqrt\delta\ (V_1^*V_1-V_2^*V_2 -V_1^*EV_2^*).
\ee

If we set $A=V_1^*V_1-V_2^*V_2$, then as $V_1^*V_1+V_2^*V_2=1$,
$A$ is a self-adjoint contraction, as required by the statement of the lemma.
Furthermore,
\[
V_1^*V_1={\frac{1+A}{2}}\qquad\text{ and }\qquad V_2^*V_2={\frac{1-A}{2}}
\]
so that $V_1$ and $V_2$ have the polar decompositions
\[
V_1=U_1\sqrt{\frac{1+A}{2}}\qquad \text{ and }\qquad V_2=U_2\sqrt{\frac{1-A}{2}}
\]
where $U_1$ and $U_2$ are Hilbert space isomorphisms. In terms of these new operators, equation  \eqref{form.50} becomes
\be\label{form.60}
C_{11}=\sqrt\delta\ \big(A - \half \sqrt{1+A} \ U_1^*EU_2\ \sqrt{1-A}\big).
\ee
Finally, if we define $F=-U_1^*EU_2$, equation \eqref{form.60} becomes equation \eqref{form.40} and as $U_1$ and $U_2$ are Hilbert space isomorphisms, equation \eqref{form.30} implies equation \eqref{form.35}.
\end{proof}

\begin{lem}\label{form.lem.20}
Let $\de \in (0,1)$ and let $\h$ be a finite dimensional Hilbert space. If $T \in \w_\delta (\h)$ is generic for $K_\de$, then there exist a pair of contractions $A,Y \in \b(\h)$ such that
$A$ is self-adjoint and
\be\label{form.70}
T=2\sqrt\delta A + (1-\delta)\sqrt{1+A}\ Y\sqrt{1-A}.
\ee
\end{lem}
\begin{proof}
Assume that $T \in \w_\delta (\h)$ is generic for $K_\de$.
By Theorem \ref{thm.10}, there exists a $(\delta,\h)$-germinator $C$ such that \eqref{par.70} holds. In particular,
\[
T=2C_{11}.
\]
By Lemma \ref{form.lem.15}, there exist $A,F \in \b(\h)$ such that $A$ is a self-adjoint contraction such that \eqref{form.35} and \eqref{form.40} hold.
If we define 
\[
Y=\frac{\sqrt\delta}{1-\delta} F,
\]
then $\| F \| \le 1$ and the equation \eqref{form.70} holds. \color{black}
\end{proof}

For $c\in \c$ and $r\ge 0$ we let $D_c(r)$ denote the closed disc in the complex plane centered at $c$ of radius $r$, that is,
\[
D_c(r)=\set{z\in \c}{|z-c|\le r}.
\]
The following lemma describes the radii of the maximal discs in $K_\delta$ centered on the major axis of $K_\delta$.
\begin{lem}\label{form.lem.30}
Let $\de\in (0,1)$.
If, for $t\in [-(1+\delta),1+\delta]$, we define
\[
\rho(t) = \sup\ \set{r}{D_{(t,0)}(r)\subseteq K_\delta},
\]
then
\[
\rho(t)=
\twopartdef{(1-\delta)\sqrt{1-\frac{t^2}{4\delta}}}
{t\in[-\frac{4\delta}{1+\delta},\frac{4\delta}{1+\delta}\ ]}
{1+\delta-|t|}{t\in [-(1+\delta),1+\delta]\setminus [-\frac{4\delta}{1+\delta},\frac{4\delta}{1+\delta}\ ]}.
\]
\end{lem}
\begin{proof}
For $t \in [-\frac{4\delta}{1+\delta},\frac{4\delta}{1+\delta}\ ]$, the maximal disc always touches the boundary of $K_\delta$ in two distinct points symmetric about the $x$-axis and the formula for $\rho(t)$ can be justified by Lagrange multipliers. For $t\not\in [-\frac{4\delta}{1+\delta},\frac{4\delta}{1+\delta}\ ]$, the maximal disc touches the boundary of $K_\delta$ in a single point $\pm(1+\delta)$, which explains the formula for $\rho(t)$ in that case.
\end{proof}
It was slightly surprising to the authors that the role played by the transition points $\pm\frac{4\delta}{1+\delta}$ in the preceding lemma was not played by the foci ($\pm2\sqrt\delta$). However, the foci of $K_\delta$ appear in the following corollary to Lemma \ref{form.lem.30}.
\begin{lem}\label{form.lem.40}  Let $\de\in (0,1)$.
If $t\in [-2\sqrt\delta,2\sqrt\delta]$ and $z\in \c$, then
\[
|z|\le (1-\delta)\sqrt{1-\frac{t^2}{4\delta}}\implies
t+z \in K_\delta.
\]
\end{lem}
\begin{proof}
As $\pm\frac{4\delta}{1+\delta} \in [-2\sqrt\delta,2\sqrt\delta] $, there are two cases, both of which follow immediately from Lemma \ref{form.lem.30}.
\end{proof}
\begin{prop}\label{form.prop.10} Let $\de \in (0,1)$.
If $\h$ is a Hilbert space (not assumed to be finite dimensional), $A$ and $Y$ are contractions in  $\b(\h)$, $A$ is self-adjoint, and $T\in \b(\h)$ is defined by the formula,
\[
T=2\sqrt\delta A + (1-\delta)\sqrt{1+A}\ Y\sqrt{1-A},
\]
then $W(T) \subseteq K_\delta$.
\end{prop}
\begin{proof}
Fix $v\in \h$ with $\norm{v}=1$. We need to show that
\be\label{form.80}
\ip{Tv}{v} \in K_\delta.
\ee
We observe that if we set $a=\ip{Av}{v}$
\begin{align*}
\ip{Tv}{v}&=2\sqrt\delta a +
(1-\delta)\ip{\sqrt{1+A}\ Y\sqrt{1-A}v}{v}\\
&=2\sqrt\delta a +
(1-\delta)\ip{Y\sqrt{1-A}v}{\sqrt{1+A}v}.
\end{align*}
Hence, as $Y$ is a contraction,
\begin{align*}
|\ip{Tv}{v}-2\sqrt\delta a|^2&=(1-\delta)^2\ \big|\ip{Y\sqrt{1-A}v}{\sqrt{1+A}v}\big|^2\\ \\
&\le (1-\delta)^2\ \norm{\sqrt{1-A}v}^2\  \norm{\sqrt{1+A}v}^2\\ \\
&=(1-\delta)^2\ \ip{(1-A)v}{v}\  \ip{(1+A)v}{v}\\ \\
&=(1-\delta)^2(1-a^2).
\end{align*}
Therefore,
\[
\ip{Tv}{v} \in D_{2\sqrt\delta a}(\ (1-\delta)\sqrt{1-a^2}\ ),
\]
which implies via Lemma \ref{form.lem.40} (with $t=2\sqrt\delta a$) that equation \eqref{form.80} holds.
\end{proof}
We close this section by noting that, in the finite dimensional case, Lemma \ref{form.lem.20} and Proposition \ref{form.prop.10} can be combined to yield the following generalization of a theorem of Ando that parametrizes the operators with numerical range in $\d^-$ (\cite[Theorem  1]{Ando}).
\begin{thm}\label{Ando-ellipse} Let $\de \in (0,1)$.
Let $\h$ be a finite dimensional Hilbert space and assume that $T\in \b(\h)$. $W(T)\subseteq K_\delta$ if and only if there exist a pair of contractions $A,Y \in \b(\h)$ such that
$A$ is self-adjoint and
\be\label{form.70.thm}
T=2\sqrt\delta A + (1-\delta)\sqrt{1+A}\ Y\sqrt{1-A}.
\ee
\end{thm}
\begin{proof}
The theorem follows from Lemma \ref{form.lem.20} and Proposition \ref{form.prop.10} provided only that we may remove the genericity assumption in Lemma \ref{form.lem.20}. Accordingly assume $T\in\w_\delta(\h)$. Choose a generic sequence $\{T_k\}$ such that $T_k\in \w_\delta(\h)$ for all $k$ and $T_k \to T$ as $k\to \infty$ in the weak operator topology. By Lemma \ref{form.lem.20}, for each $k$ there exist contractions $A_k,Y_k \in \b(\h)$ such that $A_k$ is self-adjoint and
\be\label{form.90}
T_k=2\sqrt\delta A_k + 2(1-\delta)\sqrt{\frac{1+A_k}{2}}\ Y_k\sqrt{\frac{1-A_k}{2}}.
\ee
By passing to a subsequence if necessary we may assume that there exist $A,Y\in \b(\h)$ such that
\be\label{form.100}
A_k \to A\qquad \text{and}\;\; Y_k \to Y \;\;\; \text{as} \;\; k \to \infty.
\ee
 Evidently $A_k$ and $Y_k$ contractive for all $k$ imply that $A$ and $Y$ are contractions. Similarly, as $A_k$ is self-adjoint for all $k$, $A$ is self-adjoint. Finally, equation \eqref{form.90} and relations \eqref{form.100}  imply equation \eqref{form.70.thm}.
\end{proof}

\begin{rem}
Note that if we take $\de=0$ in equation \ref{form.70.thm} then
the statement of Theorem \ref{Ando-ellipse} formally becomes the following. Let  $T\in \b(\h)$ where $\h$ is a finite dimensional Hilbert space. Then $W(T)\subseteq \ov{\d}$ if and only if there exist a pair of contractions $A,Y \in \b(\h)$ such that  $A$ is self-adjoint and
\[
T= \sqrt{1+A} \; Y \: \sqrt{1-A}.
\]
This is a true statement, by Ando's Theorem \cite[Theorem 1]{Ando}. 
We can therefore consider Theorem \ref{Ando-ellipse} as a generalization of Ando's theorem   to operators $T$ with $W(T)\subseteq K_\delta$ in the finite-dimensional case.
\end{rem}

\section{Operators $T$ with $W(T)\subseteq K_\de$  and Douglas-Paulsen operators}\label{Douglas-Paulsen}

In this section we present yet another consequence of Theorem \ref{thm.10}, a tight connection between the operators with numerical range in an ellipse and Douglas-Paulsen operators. 
 Consider a number $\de \in (0,1)$.  
  In \cite{dp86} R. G. Douglas and V. Paulsen developed a dilation theory for operators $T$
  on some Hilbert space $\h$  such that $\de^2\leq T^*T \leq 1$.  Accordingly, we shall call any such operator a {\em Douglas-Paulsen operator with parameter $\de$.}  
 
More exposition and references on such operators can be found in \cite[Chapter 9]{amy20}.

\begin{prop}\label{dp.prop.10}
Let $\de \in (0,1)$.
Fix a finite-dimensional Hilbert space $\h$ and let $X \in \b(\h)$ be a Douglas-Paulsen operator with parameter $\de$. Then $W(\pi(X))\subseteq K_\de$.
\end{prop}
\begin{proof}
Let $\h$ be $n$-dimensional and notice that $\sigma(X)\subseteq R_\de^-$. Set $T=\pi(X)=X+\de/X$. We first prove the proposition under the assumptions that 
\be\label{dp.10}
\sigma(X)\cap \sigma(\delta/X) =\emptyset
\ee
and both $X$ and $\de/X$ have $n$ distinct eigenvalues.
The hypotheses imply that  $(X,\delta/X)$ is a generic prepair for $T$.
We observe that the fact 
\[
\sigma(T) = \sigma(\pi(X)) = \pi(\sigma(X)) \subseteq \pi(R_\de^-) \subseteq K_\de
\] 
ensures (with the aid of Lemma \ref{ell.lem.20}) that $T$ is generic  for $K_\de$.
 Furthermore, as
\[
T=X+\delta/X\qquad 
\]
and $X- \de/X$ is an operator that commutes with $T$ and satisfies $(X-\de/X)^2 =T^2-4\de$, we may choose $Q=X-\de/X$ in the definition of $Z_T$, equation \eqref{ell.170}, and we have
\begin{align*}
Z_T &=\frac12
\begin{bmatrix}T&Q\\Q&T\end{bmatrix}\\ \\
&=\frac12 \begin{bmatrix}X+\delta/X&X-\delta/X\\X-\delta/X&X+\delta/X\end{bmatrix}\\ \\
&=U\begin{bmatrix}X&0\\0&\delta/X\end{bmatrix}U,
\end{align*}
where
\[
U=\frac{1}{\sqrt{2}} \bbm 1 & 1\\ 1 & -1 \ebm.
\]
Consequently, as $\norm{X},\norm{\delta/X}\le 1$ and $U$ is unitary, $\|Z_T\|\leq 1$. Hence 
condition \eqref{par.10} holds (with $S=1$) and Theorem \ref{par.thm.10} implies that $W(T)\subseteq K_\delta$.

To relax the assumptions \eqref{dp.10} and the assumption that $X$ has $n$ distinct eigenvalues, assume that $X$ is now a general Douglas-Paulsen operator with parameter $\de$.  Choose a sequence $\{X_k\}$ such that $X_k$ is a Douglas-Paulsen operator with parameter $\de$, $X_k$ has $n$ distinct eigenvalues, the stated assumptions hold for all $k$ and $X_k \to X$ as $k\to \infty$ for the weak operator topology. 
   By the special case that has been proved above, $W(\pi(X_k)) \subseteq K_\delta$ for all $k$. Hence $W(\pi(X)) \subseteq K_\de$.
\end{proof}
Proposition \ref{dp.prop.10} asserts that $\pi: X\mapsto X+\de/X$ maps the class of Douglas-Paulsen operators  with parameter $\delta$ on a Hilbert space $\h$  into the class of operators on $\h$ whose numerical ranges are contained in $K_\delta$.

\begin{fact} \label{10.3}
 The map $\pi$ is not onto, even in the case of a $2$-dimensional $\h$. 
 \end{fact}
 \begin{proof}
Consider the operator
\[
T=2\bbm  \sqrt{\de} & 1-\de \\ 0 & -\sqrt{\de} \ebm
\]
on $\c^2$. Explicit calculation in Example 3 of \cite[Section 1]{GuRao97}) shows that the 
numerical range of $T$ is $K_\de$, and so $W(T) \subseteq K_\de$.  Note that $T$ has eigenvalues $\mu_1=2\sqrt{\de}, \mu_2=-2\sqrt{\de}$, with corresponding eigenvectors $v_1=\bpm 1\\0 \epm, v_2=\bpm 2\\\sqrt{\de}-\frac{1}{\sqrt{\de}} \epm$. 

Suppose that $T=\pi(X)=X+\de X\inv$ for some $X\in\b(\c^2)$. Let us show that $X$ is not a Douglas-Paulsen operator with parameter $\de$.  By the Spectral Mapping Theorem, the eigenvalues $\lam_1,\lam_2$ of $X$ satisfy 
 $$\{\pi(\lam_1), \pi(\lam_2)\} = \{\mu_1,\mu_2\} = \{ 2\sqrt{\de},-2\sqrt{\de} \}.$$
Thus  $X$ has eigenvalues $\lam_1=\sqrt{\de}, \lam_2=-\sqrt{\de}$, which are the unique preimages of $\mu_1,\mu_2$ respectively under $\pi$, and corresponding eigenvectors $v_1,v_2$.  Therefore
\begin{align*}
X\bbm v_1 & v_2 \ebm &= \bbm v_1 & v_2 \ebm \bbm \sqrt{\de} & 0\\0 & -\sqrt{\de} \ebm, \; \text{ and so}\\
 X &= \bbm 1&2\\0 & \sqrt{\de}- \frac{1}{\sqrt{\de}} \ebm \bbm \sqrt{\de}&0\\0& -\sqrt{\de} \ebm \bbm 1 & 2\\0 & \sqrt{\de}- \frac{1}{\sqrt{\de}} \ebm\inv\\
    &= \bbm \sqrt{\de} & -2\sqrt{\de}\\ 0 & 1-\de \ebm \bbm \sqrt{\de} & 0\\0 & -\sqrt{\de} \ebm \bbm \sqrt{\de}- \frac{1}{\sqrt{\de}}  & -2 \\ 0 & 1 \ebm \frac{1}{\sqrt{\de}- \frac{1}{\sqrt{\de}} } \\
    &=\frac{\sqrt{\de}}{\de -1} \bbm \de-1 & -4\sqrt{\de} \\ 0 & 1-\de \ebm \\
    &= \bbm \sqrt{\de} & \frac{4\de}{1-\de} \\ 0 & -\sqrt{\de} \ebm.
\end{align*}
Now $\|X\| \leq 1$ if and only if $X^*X \leq 1$, which is so if and only if
$1-6\de + \de^2 \geq 0$. The last inequality is false for example if $\de=\half$, so that in this case $X$ is not a Douglas-Paulsen operator with parameter $\de$, and so
$T$ is not in the range of $\pi$.
\end{proof}

Thus, not every operator with numerical range in $K_\delta$ is in the range of $\pi$ acting on the class of Douglas-Paulsen operators with parameter $\de$. However, every operator with numerical range in $K_\delta$ has an {\em extension} to an operator of the form $\pi(X)$ for some Douglas-Paulsen operator $X$ with parameter $\de$.

\begin{prop}\label{dp.prop.20}
Let $\h$ be a finite dimensional Hilbert space and fix
$T \in \w_\delta(\h)$, where $\de\in(0,1)$. There exists $X \in  \b(\h\oplus \h)$ such that $\|X\|\leq 1$ and $\|X\inv\| \leq 1/\de$ and
\be\label{dp.20}
T=(X+\de X\inv)|\ \h \oplus \{0\}.
\ee
\end{prop}

\begin{proof}
As in the proof of Proposition \ref{dp.prop.10}, the general case follows from the generic case. Accordingly, assume that $\h$ is a finite dimensional Hilbert space and fix generic $T \in \w_\delta(\h)$. By Proposition \ref{thm.10}, there exists a generic contraction $C\in \b(\h \oplus \h)$ satisfying equations \eqref{par.50} and \eqref{par.70}. 

Let $X=C$. Note that $\norm{X} \le 1$ and also, that equation \eqref{par.50} implies that $\norm{\delta/X} \le 1$.  To see that equation \eqref{dp.20} holds observe using equation \eqref{par.70} that
\begin{align*}
\pi(X)&=X+\delta/X=C+\delta/C\\ 
&=\begin{bmatrix}1&0\\0&S\end{bmatrix}
(Z_T+\delta/Z_T)
\begin{bmatrix}1&0\\0&S\end{bmatrix}^{-1}\\ 
&=\begin{bmatrix}T&0\\0&STS^{-1}\end{bmatrix}.
\end{align*}
Thus $\h\oplus \{0\}$ is invariant under $\pi(X)$ and $T$ is the restriction of $\pi(X)$ to $\h\oplus \{0\}$.
\end{proof}

Propositions \ref{dp.prop.10} and \ref{dp.prop.20} combine to give the following characterization.
\begin{thm}\label{conn-to-D-P-oper}
Let $\h$ be a finite-dimensional Hilbert space, let $T\in\b(\h)$ and let $\de\in(0,1)$.
$W(T)\subseteq K_\de$ if and only if there exists  $X \in  \b(\h\oplus \h)$ such that $\|X\|\leq 1$ and $\|X\inv\| \leq 1/\de$ such that
\be\label{1.23}
T=(X+\de X\inv)| \h\oplus\{0\}.
\ee
\end{thm}
\begin{proof}  Suppose there exists $X\in \b(\h\oplus\h)$ such that $\h\oplus\{0\}$ is invariant for $X+\de X\inv$ and the relation \eqref{1.23} holds.  By Proposition \ref{dp.prop.10}, $W(\pi(X))\subseteq K_\de$.  Since $T$ is a restriction of $\pi(X)$, $W(T)\subseteq W(\pi(X)) \subseteq K_\de$.

Conversely, if $W(T)\subseteq K_\de$, by Proposition \ref{dp.prop.20} there exists a Douglas-Paulsen operator $X$ on $\h\oplus\h$ with parameter $\de$ such that 
$T=(X+\de X\inv)| \h\oplus\{0\}$.
\end{proof}

\section {The $\mathrm{bfd}$ norm on $\hol(G_\de)$ and the $\mathrm{dp}$ norm on $\hol(R_\de)$}\label{norms-BFD-DP}

 Following \cite[Chapter 9]{amy20}, for any $\de \in (0,1)$,  we define the {\em Douglas-Paulsen family with parameter $\de$} to be the class $\f_{\mathrm {dp}}(\de)$ of all Douglas-Paulsen operators $T$ with parameter $\de$ such that $\sigma(T) \subseteq R_\de$ and consider
 the associated calcular norm 
\be\label{dpnorm-main}
\|\phi\|_{\mathrm {dp}}= \sup_{T\in \f_{\mathrm {dp}}(\de)} \|\phi(T)\| \;\;\;\text{for } \; \phi\in \hol(R_\de).
\ee
Also  following \cite[Chapter 9]{amy20} we define the {\em B. and F. Delyon family} $\f_{\mathrm {bfd}}(G_\de)$ corresponding to the elliptical region $G_\de$ to be the class of operators $T$ such that  $\overline{W(T)} \subseteq G_\de$. 
By \cite[Theorem 1.2-1]{GuRao97}, the spectrum $\sigma(T)$  of an operator $T$ is contained in $\overline{W(T)}$, and so, by the Riesz-Dunford functional calculus, $\ph(T)$ is defined for all $\phi\in \hol(G_\de)$ and $T\in \f_{\mathrm {bfd}}(G_\de)$. 
Consider the calcular norm   
\be \label{bfdnorm-main}
\norm{\phi}_{\mathrm {bfd}}
=\sup_{T\in \f_{\mathrm {bfd}}(G_\de)}\norm{\phi(T)} \;\;\;\text{for } \; \phi \in \hol(G_\delta).
\ee
We note  that these two quantities in equations \eqref{dpnorm-main}
and \eqref{bfdnorm-main} may be infinite.
The results of Section \ref{Douglas-Paulsen}, which describe an intimate connection between the families
$\f_{\mathrm {dp}}(\de)$ and  $\f_{\mathrm {bfd}}(G_\de)$, give rise to a relationship between  
 the norms $\|\cdot\|_{\mathrm {dp}}$ and $\norm{\cdot}_{\mathrm {bfd}}$,
and between the associated Banach algebras   
\[
\hinf_{\mathrm {dp}}(R_\de)=\set{\phi \in \hol(R_\delta)}{\norm{\phi}_{\mathrm {dp}}<\infty}
\]
and
\[
\hinf_{\mathrm {bfd}}(G_\de)=\set{\phi \in \hol(G_\delta)}{\norm{\phi}_{\mathrm {\mathrm {bfd}}} <\infty}.
\]

This relationship is formalised in Theorem \ref{dp.thm.10} below.
The proof of Theorem \ref{dp.thm.10} will require a number of ingredients.
The first one is contained in
 Propositions \ref{dpmatrices}
and \ref{bfdmatrices} which show that, instead of taking the supremum in the foregoing formulae, equations   \eqref{dpnorm-main} and \eqref{bfdnorm-main},
over {\em all} operators in $\f_{\mathrm {dp}}(\de)$ and 
$\f_{\mathrm {bfd}}(G_\de)$ respectively, it is sufficient to take the supremum only over the  operators in the corresponding families that act on {\em finite-dimensional} Hilbert spaces.

A second ingredient that we shall require for the proof of Proposition \ref{bfdmatrices} is the notion of the   
 {\em Badea-Beckermann-Crouzeix family} corresponding to a pair $(c,r)$, where, for some positive integer $m$, $c=(c_1,\dots,c_m)$ and $r=(r_1,\dots,r_m)$, each $c_j \in\c$ and each $r_j > 0$.  Such a pair $(c,r)$ is called (in \cite[Section 9.6]{amy20}) a {\em bbc pair}, and corresponding to such a pair the {\em bbc family} $\f_{\mathrm{bbc}}(c,r)$ is defined to be the family of operators $T$ such that the spectrum $\sigma(T)$ is contained in the set
\[
D(c,r) \df \{z\in\c: |z-c_j| < r_j \ \text{for} \ j=1,\dots,m\}
\]
and such that
\[
\|T-c_j\|< r_j \ \text{for} \ j=1,\dots,m.
\]
With this family we associate the calcular norm $\|\cdot\|_{\f_{\mathrm{bbc}}(c,r)}$ on $\hol(D(c,r))$ defined by
\[
\|\phi\|_{\mathrm{bbc}} = \sup_{T\in \f_{\mathrm{bbc}}(c,r)} \|\phi(T)\|\;\text{for } \; \phi\in \hol(D(c,r)),
\]
which we call the {\em bbc norm} or {\em Badea-Beckermann-Crouzeix norm} with pair $(c,r)$.

A final ingredient is a useful result from \cite{am2015}, stated below as Theorem \ref{amthm1}, which sheds light   
 on the dp and bbc norms -- see Proposition \ref{dpmatrices} below.  For the convenience of the reader we recall  from \cite{am2015}
that, corresponding to a domain of holomorphy $U$ in $\c^d$ and an $m$-tuple $\ga=(\ga_1,\dots,\ga_m)$ of holomorphic functions on $U$ we denote\footnote{Here we vary the notation of \cite{am2015} in order to avoid conflict with the notation $G_\de$ that we are using in this paper.} by $E_\ga$ the domain 
 \[
 E_\ga\df\{z\in U:|\ga_j(z)|<1\  \text{for} \ j=1,\dots,m\},
 \]
 and by $\f_\ga$ the family of commuting $d$-tuples $T=(T_1,\dots,T_d)$ of Hilbert space operators such that $\sigma(T)\subseteq E_\ga$ and $\|\ga_j(T)\|\leq 1$ for $j=1,\dots,m$, where $\sigma(T)$ denotes the Taylor spectrum of the $d$-tuple $T$ of commuting operators. 
 Using $\f_\ga$ one defines a norm $\|\cdot\|_\ga$ on $\hol(E_\ga)$ thus: for any $\phi\in\hol(E_\ga)$,
 \[
  \|\phi\|_\ga = \sup\{ \|\phi(T)\|:T\in\f_\ga\},
  \] 
  and then one defines 
   $H^\infty_\ga$  to be the set of functions $\phi\in\hol(E_\ga)$ such that $\|\phi\|_\ga < \infty$.  Here, $\phi(T)$ is defined by the Taylor functional calculus (though in this paper we only require the case that $d=1$, so that the classical functional calculus for operators is sufficient).

Two examples of families expressible in the form $\f_\ga$ that we shall exploit are the following.  
   \begin{ex} \label{Fdp=Fgamma} \rm  The Douglas-Paulsen family with parameter $\de$ is an example of $\f_\ga$ for a suitable choice of $\ga$.  Indeed,  if we choose $m=2$, $U=\c\setminus\{0\}$, $\ga_1(z)=z$ and $\ga_2(z)=\de/z$ then 
\[
E_\ga=R_\de \; \text{and} \; \f_\ga = \f_{\mathrm{dp}}(\de).
\]
\hfill $\Box$
   \end{ex}
  
   \begin{ex} \label{Fbbc=Fgamma} \rm  The  bbc family $\f_{\mathrm{bbc}}(c,r)$ with parameter $(c,r)$ is another example of $\f_\ga$ for a suitable choice of $\ga$.  Indeed,    
   for the positive integer $d$ as above, if $U=\c$ and $\ga_j(z)=\frac{z-c_j}{r_j}$ for $j=1,\dots,m$ then 
\[
E_\ga\df\{z\in U:|\ga_j(z)|<1\  \text{for} \ j=1,\dots,m\}= D(c,r)
 \; \text{and} \; \f_\ga = \f_{\mathrm{bbc}}(c,r).
\]
\hfill $\Box$
   \end{ex}
 
   The norm $\|\cdot\|_\ga$ is in general quite complex and difficult to analyse, especially in several variables, depending as it does on the Taylor spectrum and the Taylor functional calculus.  To make the norm more amenable to analysis, the authors of  \cite{am2015} introduced a more elementary approach to the norm  $\|\cdot\|_\ga$ which bypasses the use of the Taylor spectrum and calculus, making use of diagonalizable matrices rather than general operators on Hilbert space, thereby making the notions of spectrum and functional calculus much simpler.
    To this end the authors  of \cite{am2015}  introduced  the family of generic operators belonging to $\f_\ga$ and the norm
 $\|\cdot\|_{\ga,\mathrm{gen}}$, defined as follows.

  A $d$-tuple $T$ of pairwise commuting operators on an $n$-dimensional Hilbert space $\h$ is said to be {\em generic} if there exist  $n$ linearly independent joint eigenvectors of $(T_1,\dots,T_d)$ whose corresponding joint eigenvalues are distinct, or equivalently, if $(T_1,\dots,T_d)$ are jointly diagonalizable and $\sigma(T)$ has cardinality $n$.
   Denote by $\f_{\ga,\mathrm{gen}}$ the family of generic commuting $d$-tuples of
  operators belonging to $\f_\ga$, and define, for $\phi\in\hol(E_\ga)$,
\[
\|\phi\|_{\ga,\mathrm{gen}} = \sup_{T\in\f_{\ga,\mathrm{gen}}} \|\phi(T)\|.
\]
Evidently
\be\label{easy}
\|\phi\|_{\ga,\mathrm{gen}} = \sup_{T\in\f_{\ga,\mathrm{gen}}} \|\phi(T)\|
\leq \sup_{T\in\f_{\ga}} \|\phi(T)\| =
\|\phi\|_\ga.
\ee
One of the main results of \cite{am2015} implies that in fact  the inequality \eqref{easy} holds with equality.  We show this fact in Corollary \ref{matrices} below.

  The following theorem is a special case  of \cite[Theorem 1]{am2015}.  The statement makes use of the ``polydisc norm" $\|\cdot\|_{H^\infty_m}$ on $\hol(\d^m)$\footnote{The norm $\|\cdot\|_{H^\infty_m}$ is another example of $\|\cdot\|_\ga$, this time with $\ga=(z_1,\dots,z_d)$.}, often also called the Schur-Agler norm and defined for $F\in\hol(\d^d)$ by
  \begin{align*}
  \|F\|_{H^\infty_m} \df  \sup \{\|F(T)\|:  & \; T \; \text{is a commuting $m$-tuple of contractions $T=(T_1,\dots,T_m)$} \\
   &\; \text{such that } \; \sigma(T)\subseteq \d^m\}.
  \end{align*}
 Then $H^\infty_m$ is defined to be the set of functions $\phi\in\hol(\d^m)$ such that $\|\phi\|_{H^\infty_m} < \infty$.  For more about $H^\infty_m$ see \cite[Section 9.7]{amy20}.
 
  A useful relationship between $H^\infty_\ga$  and $H^\infty_m$
 is revealed in the following theorem.
 
\begin{thm}\label{amthm1} {\rm \cite[Theorem 1]{am2015}.}
Let $U$ be an open set in $\c^d$ and let $\ga=(\ga_1,\dots,\ga_m)$ be a $m$-tuple of analytic functions on $U$.
For any $\phi\in\hol(E_\ga), \|\phi\|_{\ga,\mathrm{gen}} \leq 1$ if and only if there exists a  function $F\in\hol(\d^m)$ such that  $\|F\|_{H^\infty_m}\leq 1$ and $\phi(z)=F\circ \ga(z)$ for all $z\in E_\ga$.
\end{thm}

A simple consequence of Theorem \ref{amthm1} is the following statement.
\begin{cor}\label{matrices}
Let $U$ be an open set in $\c^d$ and let $\ga=(\ga_1,\dots,\ga_m)$ be an $m$-tuple of analytic functions on $U$.
For any $\phi\in\hol(E_\ga)$,
\be\label{useful}
\|\phi\|_{\ga,\mathrm{gen}} = \|\phi\|_\ga.
\ee
\end{cor}
\begin{proof}
We have noted the inequality \eqref{easy} above.

Now let us prove the reverse of the inequality \eqref{easy}.  Consider any function $\phi\in\hol(E_\ga)$ and let $\psi=\phi/\|\phi\|_{\ga,\mathrm{gen}}$, so that $\psi\in\hol(E_\ga)$ and $\|\psi\|_{\ga,\mathrm{gen}} =1$.  By Theorem \ref{amthm1},
there exists $F\in\hol(\d^m)$ such that  $\|F\|_{H^\infty_m}\leq 1$ and $\psi(z)=F\circ \ga(z)$ for all $z\in E_\ga$.

Consider any operator $T\in\f_\ga$.  Then $\sigma(T)\subseteq E_\ga$, and so, by the spectral mapping theorem for the Taylor spectrum,
\[
\sigma(\ga(T))=\ga(\sigma(T))\subseteq \ga(E_\ga) \subseteq \d^m.
\]
Moreover $\ga(T)=(\ga_1(T),\dots,\ga_m(T))$ is a commuting $m$-tuple of contractions.  Hence, since $\|F\|_{H^\infty_m}\leq 1$, $\|F(\ga(T))\|\leq 1$, that is to say $\|\psi(T)\|\leq 1$.
Taking the supremum over $T\in\f_\ga$ in this inequality we find that $\|\psi\|_\ga \leq 1$, and so $\|\phi\|_\ga/\|\phi\|_{\ga,\mathrm{gen}}\leq 1$.  Thus the equation \eqref{useful} holds.
\end{proof}

\begin{prop}\label{dpmatrices}
{\rm (i)}. Let $\de\in (0,1)$ and let $\phi\in\hol(R_\de)$. The following equation holds:
\be\label{usefuldp} 
\|\phi\|_{\mathrm {dp}} = \sup_{T\in \f_{\mathrm {dp}}(\de), \; T\text{ is a matrix}} \|\phi(T)\|.
\ee
{\rm (ii)}. Let $c=(c_1,\dots,c_m)$ and $r=(r_1,\dots,r_m)$, where $c_j\in\c$ and $r_j>0$ for each $j$, and let
\[
D(c,r) \df \{z\in\c: |z-c_j| < r_j \ \text{for} \ j=1,\dots,m\}.
\]
Let $\phi\in\hol(D(c,r))$. The following equation holds:
\be\label{usefulbbc} 
\|\phi\|_{\mathrm {bbc}} = \sup_{T\in \f_{\mathrm {bbc}}(c,r), \; T\text{ is a matrix}} \|\phi(T)\|.
\ee
\end{prop}

\begin{proof}
(i). Clearly
\be\label{obvious}
\|\phi\|_{\mathrm {dp}} \geq \sup_{T\in \f_{\mathrm {dp}}(\de), \; T\text{ is a matrix}} \|\phi(T)\|.
\ee
To prove the reverse inequality, it suffices to observe that if we choose  $U=\c\setminus\{0\}$, $\ga_1(z)=z$ and $\ga_2(z)=\de/z$ then 
\[
E_\ga=R_\de \; \text{and} \; \f_\ga = \f_{\mathrm{dp}}(\de).
\]
Therefore, by Corollary \ref{matrices}, for any $\phi\in\hol(R_\de)$
\begin{align*}
\|\phi\|_{\mathrm {dp}} &= \|\phi\|_\ga = \|\phi\|_{\ga,gen} \\
	& \leq \sup_{T\in \f_{\mathrm {dp}}(\de), \; T\text{ is a matrix}} \|\phi(T)\|.
\end{align*}
Thus the equality \eqref{usefuldp} is true.

(ii)
As we have shown in Example \ref{Fbbc=Fgamma}, for $\ga=(\ga_1,\dots,\ga_m)$ where $\ga_j$ is the polynomial $(z-c_j)/r_j$ for $j=1,\dots,m$,
\[
\f_\ga = \f_{\mathrm{bbc}}(c,r).
\]
The rest of the proof  follows from Corollary \ref{matrices} and is similar to the case (i).
\end{proof}

Our goal is now to prove Proposition \ref{bfdmatrices} which shows the supremum in the definition of $\|\phi\|_{\mathrm{bfd}}$, instead of being taken over all operators in the B. and F. Delyon family, can equivalently be taken only over the operators in the family that act on a finite-dimensional Hilbert space.
Since the family $\f_{\mathrm{bfd}}(G_\de)$ is not of the form $\f_\ga$ for any $\ga$, we cannot deduce the desired equality directly from Corollary \ref{matrices}.  We therefore approximate the bfd norm by the bbc norm corresponding to a suitable bbc pair, and then use  Proposition  \ref{dpmatrices} to show that the required inequality holds. 

The approximation result we need, which is \cite[Lemma 9.76]{amy20}, is the following.
\begin{lem}\label{bbclem20}
Let $C$ be a smoothly bounded open strictly convex set in $\c$ and $\eps>0$. If $\h$ is a Hilbert space, $T\in \b(\h)$, and $\overline{W(T)} \subseteq C$,  then there exists a bbc pair $(c,r)$ such that
\be\label{bbc35}
C \subseteq D(c,r) \subseteq C+\eps\d
\ee
and
\be\label{bbc36}
T\in \f_{\mathrm{bbc}}(c,r).
\ee
\end{lem}
 It is easy to see that $G_\de$ is indeed  a smoothly bounded open strictly convex set in $\c$.

\begin{prop}\label{bfdmatrices}
Let $\de\in (0,1)$ and let $\phi\in\hol(G_\de)$. The following equation holds:
\be\label{usefulbfd}
\|\phi\|_{\mathrm {bfd}} = \sup_{T\in \f_{\mathrm {bfd}}(G_\de), \; T\text{ is a matrix}} \|\phi(T)\|.
\ee
\end{prop}

\begin{proof}  It is obvious that
\be\label{usefulbfd1}\notag
\sup_{T\in \f_{\mathrm {bfd}}(G_\de)}\norm{\phi(T)} \ge \sup_{T\in \f_{\mathrm {bfd}}(G_\de), \; T\text{ is a matrix}} \|\phi(T)\|.
\ee
Therefore, equation \eqref{usefulbfd} will follow if we can show that
\be\label{1}
\sup_{T\in \f_{\mathrm {bfd}}(G_\de)}\norm{\phi(T)} \le \sup_{T\in \f_{\mathrm {bfd}}(G_\de), \; T\text{ is a matrix}} \|\phi(T)\|.
\ee

Notice that
\be
\label{2}\notag
s\in [0,1] \text{ and } T\in \f_{\mathrm {bfd}}(G_\delta) \implies sT\in \f_{\mathrm {bfd}}(G_\delta).
\ee
Also, as $\phi(sz)$ converges uniformly to $\phi(z)$ on compact subsets of $G_\delta$ as $s\to 1-$, it follows by the continuity of the Riesz-Dunford Functional Calculus that
\be\label{3}\notag
\norm{\phi(T) - \phi(sT)} \to 0 \text{ as } s \to 1-
\ee
whenever $T\in \f_{\mathrm {bfd}}(G_\delta)$. These two facts imply that
\be\label{4}\notag
\sup_{T\in \f_{\mathrm {bfd}}(G_\de)}\norm{\phi(T)}=\sup_{T\in \f_{\mathrm {bfd}}(G_\de),\, s\in[0,1)}\norm{\phi(sT)}.
\ee
Consequently, inequality \eqref{1} will follow if we can show that
\be\label{5}
\sup_{T\in \f_{\mathrm {bfd}}(G_\de),\ s\in[0,1)}\norm{\phi(sT)} \le
\sup_{T\in \f_{\mathrm {bfd}}(G_\de), \; T\text{ is a matrix}} \|\phi(T)\|.
\ee
To prove inequality \eqref{5} fix $T\in \f_{\mathrm {bfd}}(G_\de)$, and let $s\in [0,1)$. Let $C=sG_\delta$. As $C^-$ is a compact subset of $G_\delta$, there exists $\eps >0$ such that
\be\label{6}
C+\eps \d \subseteq G_\delta.
\ee
Also, as  $\overline{W(T)} \subseteq G_\delta$, $\overline{W(sT)} \subseteq C$.  Consequently, by Lemma \ref{bbclem20} (with $T$ replaced by $sT$) there exists a bbc pair $(c,r)$ such that
\be\label{7}
C \subseteq D(c,r) \subseteq C+\eps \d
\ee
and
\be\label{8}
sT\in \f_{\mathrm{bbc}}(c,r).
\ee  
Therefore, by Proposition \ref{dpmatrices}, we have
\[
\|\phi(sT)\| \leq \|\phi\|_{\mathrm{bbc}} =\sup_{X\in \f_{\mathrm{bbc}}(c,r), \; X\text{ is a matrix}} \|\phi(X)\|.
\]
Hence, if $\eta>0$, there is a matrix $M\in  \f_{\mathrm{bbc}}(c,r)$ such that
\[
\|\phi(M)\| > \|\phi\|_{\mathrm{bbc}} -\eta,
\]
and so
\be\label{9}
\|\phi(sT)\| \leq \|\phi\|_{\mathrm{bbc}} < \|\phi(M)\|+\eta.
\ee   
But, as $M\in \f_{\mathrm {bbc}}(c,r)$, in particular,  $\overline{W(M)} \subseteq D(c,r)$ and the inclusions \eqref{7} and \eqref{6} imply that $D(c,r) \subseteq G_\delta$.  Therefore, $M\in \f_{\mathrm {bfd}}(G_\delta)$.

To summarize, in the previous two paragraphs we have shown that if  $T\in \f_{\mathrm {bfd}}(G_\de)$, $s\in [0,1)$, and $\eta >0$, then there exists a matrix $M\in \f_{\mathrm {bfd}}(G_\delta)$ such that the inequality \eqref{9} holds. This proves the inequality \eqref{5} and completes the proof of Proposition \ref{bfdmatrices}.
\end{proof}

Now observe that if $\phi \in \hol(G_\delta)$ then we may define $\pi^\sharp(\phi) \in \hol(R_\delta)$ by the formula
\[
\pi^\sharp(\phi)(\lambda)=\phi(\pi(\lambda)) \qquad\text{ for all } \lambda\in R_\delta.
\] 

We record the following simple fact from complex analysis without proof.
\begin{lem}\label{dp.lem.10} Let $\de \in (0,1)$ and let
 $\psi \in \hol(R_\delta)$. Then $\psi \in \ran \pi^\sharp$ if and only if $\psi$ is \emph{symmetric} with respect to the involution $\lam\mapsto\de/\lam$ of $R_\de$, that is, if and only if $\psi$ satisfies
\[
\psi(\delta/\lambda)=\psi(\lambda)
\]
for all $\lambda \in R_\delta$.
\end{lem}

The following theorem gives an intimate connection between the $\norm{\cdot}_{\mathrm {dp}}$ and $\norm{\cdot}_{\mathrm {bfd}}$ norms.
\begin{thm}\label{dp.thm.10}  Let $\de \in (0,1)$.
The mapping $\pi^\sharp$ is an isometric isomorphism from $\hinf_{\mathrm {bfd}}(G_\de)$ onto the set of symmetric functions with respect to the involution $\lam\mapsto\de/\lam$ in $\hinf_{\mathrm {dp}}(R_\de)$,  so that, for all $\phi\in \hol(G_\de)$,
\be\label{bfd=dp}
\|\phi\|_{\mathrm {bfd}} = \|\phi\circ\pi\|_{\mathrm {dp}}.
\ee
\end{thm}
\begin{proof} 
To see that the equation \eqref{bfd=dp} holds it suffices to prove 
that, for fixed $\phi \in \hol(G_\delta)$ the following inequalities are satisfied:
\be\label{dp.30}
\norm{\pi^\sharp(\phi)}_{\mathrm {dp}}\le \norm{\phi}_{\mathrm {bfd}}
\ee
and
\be\label{dp.40}
\norm{\pi^\sharp(\phi)}_{\mathrm {dp}}\ge \norm{\phi}_{\mathrm {bfd}}.
\ee

To prove the inequality \eqref{dp.30} fix  a matrix $X \in \f_{\mathrm {dp}}(\de)$.   
By Proposition \ref{dp.prop.10}, $\pi(X) \in \f_{\mathrm {bfd}}(G_\de)$.
 Hence,
\begin{align*}
\norm{\pi^\sharp(\phi)(X)}&=\norm{\phi(\pi(X))}\\ 
&\le \sup_{T\in \f_{\mathrm {bfd}}(G_\de)}\norm{\phi(T)}
=\norm{\phi}_{\mathrm {bfd}}.
\end{align*}
Consequently, by Proposition \ref{dpmatrices},
\[
\norm{\pi^\sharp(\phi)}_{\mathrm {dp}}=\sup_{X\in \f_{\mathrm {dp}}(\de),\; X \; \text{is a matrix}}\norm{\pi^\sharp(\phi)(X)}\le
\norm{\phi}_{\mathrm {bfd}}.
\]

To prove the inequality \eqref{dp.40}, that $\|\phi\|_{\mathrm {bfd}} \leq \|\phi\circ\pi\|_{\mathrm {dp}}$, we use Proposition \ref{dp.prop.20}.  Indeed, let $\phi\in \hol(G_\de)$, by Proposition  \ref{bfdmatrices}, the following equation holds:
\be\label{usefulbfd-2}
\|\phi\|_{\mathrm {bfd}} = \sup_{T\in \f_{\mathrm {bfd}}(G_\de), \; T\text{ is a matrix}} \|\phi(T)\|.
\ee
Consider any matrix $T\in\f_{\mathrm {bfd}}(G_\de)$.  We have $\overline{W(T)} \subseteq G_\de \subset K_\de$, and so Proposition \ref{dp.prop.20} applies that there exists $X \in  \b(\h\oplus \h)$ such that $\|X\|\leq 1$ and $\|X\inv\| \leq 1/\de$ and
$T=(X+\de X\inv)|\ \h \oplus \{0\}$.
Thus $T$ is a restriction of $\pi(X)$ to an invariant subspace
$\n= \h \oplus \{0\}$ for $\pi(X)$.
Therefore,
\begin{align*}
\norm{\phi(T)}&=\norm{\phi(\pi(X)|\n)}\\ 
&\le \norm{\phi(\pi(X))} =\norm{\pi^\sharp(\phi)(X)}\\ 
&\le \norm{\pi^\sharp(\phi)}_{\mathrm {dp}}.
\end{align*}
Take the supremum of both sides over all matrices $T\in\f_{\mathrm {bfd}}(G_\de)$ to infer that 
\be\label{dp.401}
\norm{\phi}_{\mathrm {bfd}}
=\sup_{T\in \f_{\mathrm {bfd}}(G_\de),\; T\; \text{is a matrix}}\norm{\phi(T)}\le\norm{\pi^\sharp(\phi)}_{\mathrm {dp}}.
\ee

We claim now that the mapping $\pi^\sharp$ maps $\hinf_{\mathrm {bfd}}(G_\de)$ surjectively onto the set of  functions in $\hinf_{\mathrm {dp}}(R_\de)$ that are symmetric with respect to the involution $\lam\mapsto\de/\lam$.
For consider any  function $f\in \hinf_{\mathrm {dp}}(R_\de)$ that is symmetric with respect to the involution $\lam\mapsto\de/\lam$, that is, such that $f(\de/\lam)=f(\lam)$ for all $\lam\in R_\de$.  By Lemma \ref{dp.lem.10}, there exists a function $\phi\in\hol(G_\de)$ such that $\pi^\sharp \phi = f$.  By the inequality \eqref{dp.401}, $\|\phi\|_{\mathrm{bfd}} \leq \|f\|_{\mathrm{dp}} < \infty$.
Hence $\phi\in \hinf_{\mathrm {bfd}}(G_\de)$ and $\pi^\sharp \phi=f$.
\end{proof}

\section{The Banach algebra $\hinf_{\mathrm {bfd}}(G_\de)$ and the symmetrized bidisc}\label{delyonnorm}

In this section we shall give an interpretation of the B. and F. Delyon norm on $\hol(G_\delta)$ in terms of the solution to a certain extremal problem in two complex variables. Let $G$ denote the symmetrized bidisc, the open set in $\c^2$ defined by either of the equivalent formulas 
\[
G=\rho(\d^2)
\]
where $\rho:\c^2 \to \c^2$, the \emph{symmetrization map}, is defined by
\[
\rho(z)=(z_1+z_2,z_1z_2),\qquad z=(z_1,z_2)\in \c^2,
\]
or
\[
G=\set{(s,p)\in \c^2}{|s-\bar s p|<1-|p|^2} \;\text{ (see \cite[Theorem 2.1]{AY04})}.
\] 
The following simple lemma demonstrates that $G_\delta$ may be viewed as a slice of $G$.
\begin{lem}\label{sym.lem.10}
For $\delta \in (0,1)$,
\[
(s,\delta) \in G \iff s\in G_\delta.
\]
\end{lem}
\begin{proof}   Observe that
\begin{align*}
(s,\de) \in G & \Leftrightarrow \;\text{there exist} \;z,w \in \d  \;\text{such that}  \;s=z+w, \de=zw\\ 
 ~ & \Leftrightarrow 
\text{ there exists }z\in \c \text{ such that } \de < |z|<1 \text{ and } s=z+\de/z \\ 
 ~ & \Leftrightarrow 
s \in \text{ the image of } R_\de \text{ under the map } z \mapsto z+\de/z \\
~  &\Leftrightarrow   s \in G_\de \;\text{ (see Lemma \ref{ell.lem.10})}.
\end{align*}
\end{proof}
Evidently, the lemma implies that if $\Phi \in \hol(G)$, then we may define $\phi \in \hol(G_\delta)$ by the formula
\be\label{sym.10}
\phi(\mu)=\Phi(\mu,\delta) \quad \text{ for all } \mu \in G_\delta.
\ee

Lemma \ref{sym.lem.10} prompts the question: if $\phi \in \hol(G_\delta)$, does there exist a function $\Phi \in \hol(G)$ such that equation \eqref{sym.10} holds? Or, in other words. such that $\ph$ is the restriction of $\Phi$ to $G_\de$?
A classical result relevant to this question is Cartan's Extension Theorem \cite{cartan}, which states that if $V$ is an analytic variety in a domain of holomorphy $U$ and $\phi$ is a holomorphic function on $V$ then there is a holomorphic function $\Phi$ on $U$ such that the restriction of $\Phi$ to $V$ is $\phi$.

The symmetrized bidisc $G$ is a domain of holomorphy in $\c^2$ and $G_\de$ is a variety in $G$, so
Cartan's theorem applies to prove that there exists a function $\Phi \in \hol(G)$ such that $\phi$ is the restriction of $\Phi$ to $G_\de$; however, Cartan's theorem gives no information about bounds on $\Phi$, even when $\phi$ is bounded on $G_\de$.  It is therefore natural to pose the following extremal problems: 
\begin{problem}
Given $\phi \in  \in H^\infty_{\mathrm {bfd}}(G_\de)$, what is the minimum $m_\phi$ of $\|\Phi\|_{H^\infty(G)}$ over all $\Phi\in H^\infty(G)$ such that $\phi$ is the restriction of $\Phi$ to $G_\de$?  For which $\phi$ is $m_\phi$ finite?
\end{problem}
The next two propositions combine to show that the minimum $m_\phi$  is $\|\phi\|_{\mathrm {bfd}}$, so that $m_\phi$ is finite if and only if $\phi\in \hinf_{\mathrm {bfd}}(G_\delta)$.

\begin{prop}\label{sym.prop.10} Let $\de \in (0,1)$. 
If $\Phi \in \hinf(G)$ and $\phi$ is defined as in equation \eqref{sym.10}, then $\phi \in \hinf_{\mathrm {bfd}}(G_\delta)$ and $\norm{\phi}_{\mathrm {bfd}} \le \norm{\Phi}_{H^\infty(G)}$.
\end{prop}
\begin{proof} 
Let $\Psi = \Phi \circ \rho$  and define $\psi:R_\delta \to \c$  by
$\psi(\lambda) = \Psi(\lambda,\delta/\lambda)$. Then Theorem 9.56 from \cite{amy20} implies that
$\psi\in \hinf_{\mathrm {dp}}(R_\de)$ and $\norm{\psi}_{\mathrm {dp}}\le \norm{\Psi}_{H^\infty(G)}$. But for $\lambda \in R_\delta$,
\begin{align*}
\psi(\lambda)&=\Psi(\lambda,\delta/\lambda)\\
&=\Phi(\lambda+\delta/\lambda,\delta)\\
&=\phi(\pi(\lambda)),
\end{align*}
so that $\psi = \pi^\sharp(\phi)$. Hence, using Theorem \ref{dp.thm.10}, we deduce that
\[
\norm{\phi}_{\mathrm {bfd}} = \norm{\psi}_{\mathrm {dp}}\le \norm{\Psi}_{H^\infty(G)} = \norm{\Phi}_{H^\infty(G)}.
\]
\end{proof}
\begin{prop}\label{sym.prop.20}
If $\phi \in \hinf_{\mathrm {bfd}}(G_\delta)$, then there exists $\Phi \in \hinf(G)$ such that $\norm{\Phi}_\infty=\norm{\phi}_{\mathrm {bfd}}$ and equation \eqref{sym.10} holds.
\end{prop}
\begin{proof}
Let 
\be\label{sym.15}
\psi=\pi^\sharp(\phi).
\ee
 By Theorem \ref{dp.thm.10}, $\psi$ is a symmetric function in $\hinf_{\mathrm {dp}}(R_\de)$ and $\norm{\psi}_{\mathrm {dp}} = \norm{\phi}_{\mathrm {bfd}}$. Hence, by
Theorem 9.56 in \cite{amy20}, there exists $\Psi_0 \in \hinf(\d^2)$ such that
\be\label{sym.20}
\norm{\Psi_0}_\infty = \norm{\phi}_{\mathrm {bfd}}
\ee
and
\be\label{sym.30}
\Psi_0(\lambda,\delta/\lambda)=\psi(\lambda),\qquad \lambda\in R_\delta.
\ee

If we define $\Psi \in \hinf(\d^2)$ by
\[
\Psi(z_1,z_2)=\frac{\Psi_0(z_1,z_2)+\Psi_0(z_2,z_1)}{2},\qquad z\in \d^2,
\]
then \eqref{sym.20} implies that
\be\label{sym.40}
\norm{\Psi}_\infty \le \norm{\phi}_{\mathrm {bfd}}
\ee
and, as $\psi$ is symmetric, \eqref{sym.30} implies that
\be\label{sym.50}
\Psi(\lambda,\delta/\lambda)=\psi(\lambda),\qquad \lambda\in R_\delta.
\ee

Now, from the definition, we see that $\Psi$ a symmetric function on $\d^2$. Therefore, by the Waring-Lagrange theorem for $\hinf(\d^2) $ \cite[Theorem 7.1]{amy20}, there exists $\Phi \in \hinf(G)$ satisfying
\color{black}
\be\label{sym.60}
\norm{\Phi}_{\infty}=\norm{\Psi}_\infty
\ee
and
\be\label{sym.70}
\Phi(z_1+z_2,z_1z_2)=\Psi(z_1,z_2),\qquad z\in \d^2.
\ee
We claim that $\Phi$ satisfies the conditions of the proposition.

We first show that \eqref{sym.10} holds. Letting $\mu=\pi(\lambda)$, we have
\begin{align*}
\Phi(\mu,\delta)&=\Phi(\lambda+\delta/\lambda,\lambda\ (\delta/\lambda))\\ \\
 &= \Psi(\lambda,\delta/\lambda) \qquad\qquad \text{by equation} \; \eqref{sym.70} \\ \\
 &=\psi(\lambda) \qquad \qquad\qquad \text{by equation} \; \eqref{sym.50} \\ \\
 &=\pi^\sharp(\phi)(\lambda)  \qquad\qquad\  \text{by equation} \; 
\eqref{sym.15} \\ \\
&=\phi(\mu).
\end{align*}
To see that $\norm{\Phi}_\infty=\norm{\phi}_{\mathrm {bfd}}$ observe that the relations \eqref{sym.40} and \eqref{sym.60} imply that $\norm{\Phi}_\infty\le \norm{\phi}_{\mathrm {bfd}}$. As equation \eqref{sym.10} holds, the reverse inequality,
$\norm{\Phi}_\infty\ge \norm{\phi}_{\mathrm {bfd}}$, follows from Proposition \ref{sym.prop.10}.
\end{proof}

Propositions \ref{sym.prop.20} and \ref{sym.prop.10} combine to yield the following statement.
\begin{thm} \label{extremalproblem-main} Let $\de\in(0,1)$.  For any $\phi\in \hol(G_\de)$, the minimum of $\|\Phi\|_{H^\infty(G)}$ over all functions $\Phi \in H^\infty(G)$ that extend the function $(s,\de)\mapsto \phi(s)$ is $\norm{\phi}_{\mathrm {bfd}}$.   In particular, an analytic function $\phi$ on $G_\de$ has a bounded extension $\Phi$ to $G$ if
and only if $\phi\in H^\infty_{\mathrm{bfd}}(G_\de)$.

\end{thm}

\section{Appendix on the numerical range and dilations} \label{dilations-history}

In this appendix we recall some known results on the numerical range $W(T)$ of an operator $T$ and in particular the relationship between the inclusion of $W(T)$ in a prescribed disc and the existence of a dilation of $T$ of a particular type.

The notion of a dilation of an operator has proved fruitful in the study of operators.
It is a crucial element of the ``harmonic analysis of operators on Hilbert space", the theory established by B. Sz.-Nagy and C. Foias in \cite{szn-foi}.  Consider an operator $T$ on a Hilbert space $\h$, and let $\h$ be decomposed into an orthogonal direct sum of three subspaces $\h_1,\h_2,\h_3$.  Represent $T$ by an operator matrix with respect to this decomposition:
\be \label{blockT}
T \sim \bbm T_{11} & T_{12} & T_{13} \\ T_{21} & T_{22} & T_{23} \\ T_{31} & T_{32} & T_{33} \ebm.
\ee
Here of course $T_{ij}$ denotes the operator $P_i T|\h_j$, where $P_i$ is the orthogonal projection from $\h$ to $\h_i$ and $\cdot |\h_j$ denotes restriction to $\h_j$.  In general the properties of the block entries $T_{ij}$ are related in quite subtle ways to the properties of the operator $T$,
but in the special case that the block matrix \eqref{blockT} is lower triangular, in the sense that the operators $T_{21}, T_{31},T_{32}$ are all zero (or equivalently, that $\h_1$ and $\h_1 \oplus \h_2$ are invariant subspaces of $\h$ for $T$), the relation is more straightforward.  Observe in particular that, if $T$ has a block upper triangular matrix with respect to the orthogonal decomposition $\h=\h_1 \oplus \h_2 \oplus \h_3$, say
\be\label{uptri}
T \sim \bbm T_{11} & T_{12} & T_{13} \\ 0 & T_{22} & T_{23} \\ 0 & 0 & T_{33} \ebm,
\ee
then, for every positive integer $n$,
\[
T^n \sim \bbm T_{11}^n & * & * \\ 0 & T_{22}^n & * \\ 0 & 0 & T_{33}^n \ebm,
\]
where the stars denote unspecified entries.  Thus, in this block triangular case, $(T_{22})^n$ is the compression of $T^n$ to $\h_2$ for all $n \geq 1$: $T_{22}^n=P_2T^n|\h_2$.  This property brings us to the terminology of {\em dilations}. 

 One says that an operator $X\in\b(\k)$ is a dilation of an operator $T\in\b(\h)$ if $\h$ is a closed subspace of $\k$ and, for every positive integer $n$, $T^n$ is the compression to $\h$ of $X^n$, or, in other words, if $T^n= P_\h X^n |\h$ for all $n\geq 1$.
The preceding paragraph shows that, if $T$ has the block upper triangular form \eqref{uptri} with respect to the orthogonal decomposition $\h =\h_1\oplus\h_2\oplus\h_3$, then $T$ is a dilation of $T_{22}$.

 According to a lemma of D. Sarason \cite[Lemma 0]{sar}, a converse statement also holds: 
 
 \begin{lem} \label{sarasonlem}
 Let $X$ be an operator on a Hilbert space $\k$ and suppose $X$ is a dilation of an operator $T$ on a closed subspace $\h$ of $\k$.  Then there exist closed subspaces $\h_1$ and $\h_3$ of $\k$ such that $\k$ is the orthogonal direct sum of $\h_1$, $\h$ and $\h_3$, $X$ has an upper triangular operator matrix with respect to the decomposition $\k= \h_1\oplus \h\oplus\h_3$ , and $T$ is the $(2,2)$ block in the operator matrix of $X$.
 \end{lem}
 In the above statement of the lemma, the triangularity assertion about $X$ is equivalent to
 both $\h_1$ and $\h_1\oplus\h$ being invariant subspaces of $\k$ for $X$.
Here is a proof (which Sarason attributes to C. Foias) of Lemma \ref{sarasonlem}.

\begin{proof}  Let $\m$ be the smallest closed $X$-invariant subspace of $\k$ containing $\h$ and let $P$ and $Q$ be the orthogonal projections from $\k$ onto $\h$ and $\m$ respectively.  Let $\h_1=\m\ominus\h$, so that $Q-P$ is the orthogonal projection onto $\h_1$.  We assert that $\h_1$ is $X$-invariant, or equivalently that
\[
(Q-P)X(Q-P) = X(Q-P).
\]
Since clearly $QXQ=XQ$ and $QXP=XP$, it is enough to show that $PXP=PXQ$.  Now for $y\in \h$ and for $m,n\geq 1$, the dilation condition implies that
\[
PX^nPX^my = PX^n T^my =T^nT^my=T^{n+m}y=PX^{n+m}y.
\]
But $\m$ is the closed span of the vectors $X^my$ with $y\in\h$ and $m\geq 1$.   We may conclude that $PX^nPx=PX^nx$ for all $x\in\m$, and so $PXP=PXQ$, as required.  Thus $\h_1$ and $\m$ are invariant subspaces for $X$.  Choice of $\h_3=\m^\perp$ completes the orthogonal decomposition $\k=\h_1\oplus\h\oplus \h_3$ with respect to which $X$ has an upper triangular operator matrix, with  $T$ in the $(2,2)$-position.
\end{proof}

In the theory of operators on Hilbert space the most important result concerning dilations is the {\em Nagy dilation theorem}, which asserts that any contraction $T$ on a Hilbert space $\h$ has a dilation $U$ on some Hilbert space $\k \supseteq \h$ such that $U$ is a unitary operator on $\k$.
There are many proofs of this result (see \cite[page 52]{szn-foi}).  One of the shortest proofs is to deduce it from the following generalization of the Herglotz Representation Theorem due to Naimark \cite{nai2,nai1}.

\begin{thm}\label{ell.thm.10.app}
Let $\h$ be a Hilbert space and assume that $V$ is an analytic $\b(\h)$-valued function defined on $\d$ satisfying $\re V(z) \ge 0$ for all $z\in \d$. Then there exist a Hilbert space $\k$, an isometry $I:\h \to \k$, and a unitary operator $U\in \b(\k)$ such that
\[
V(z) = I^*\frac{1+zU}{1-zU}I
\]
for all $z\in \d$.
\end{thm}
Indeed, if $T$ is a contraction then, for any $z\in\d$,
\[
\re \frac{1+zT}{1-zT} = 2(1-zT)\inv (1-|z|^2TT^*)(1-\bar zT^*)\inv  >0.
\]
It follows from Naimark's theorem that there exist a Hilbert space $\k$, an isometry $I:\h \to \k$, and a unitary operator $U\in \b(\k)$ such that
\[
\frac{1+zT}{1-zT} =I^*\frac{1+zU}{1-zU}I  \quad\text{ for all }z\in\d.
\]
Expanding both sides of the equation as power series in $z$, we have
\[
1+2zT+2z^2T^2 +2z^3T^3+\dots = I^*(1+2zU+2z^2U^2+2z^3U^3 +\dots)I\quad\text{ for all }z\in\d,
\]
and on equating coefficients of $z^n$ we deduce that $T^n=I^* U^n I$ for all positive integers $n$, so that $U$ is a dilation of $T$.  The lemma of Sarason proved above then gives us a geometric interpretation of the relation between $U$ and $T$: $T$ is the compression of $U$ to a semi-invariant subspace of $U$, which is to say, $T$ is the compression of $U$ to the orthogonal complement of of one invariant subspace of $U$ in another.  In the notation of the above proof of Sarason's Lemma, $T$ is the compression of $U$ to $\m\ominus \h_1$, while both $\m$ and $\h_1$ are invariant subspaces for $U$, which makes $\h=\m\ominus \h_1$ a semi-invariant subspace of $U$.


\begin{thebibliography}{99}

\bibitem{am2015} J. Agler and J. E. McCarthy, Operator theory and the Oka extension theorem, {\em Hiroshima Math. J.} {\bf 45} (2015) 9-34.

\bibitem{amy20} J. Agler, J. E. McCarthy and N. J. Young, {\em Operator Analysis: Hilbert space methods in complex analysis}, Cambridge Tracts in Mathematics {\bf 219}, Cambridge University Press, Cambridge, U.K., 2020.


\bibitem{AY04} J. Agler and N. J. Young, The hyperbolic geometry of the symmetrized bidisc, {\em J. Geom. Anal.} {\bf 14} (2004) 375-403.

\bibitem{Ando} T. Ando,  Structure of operators with numerical radius one,  {\em Acta Sci. Math. (Szeged)} {\bf 34} (1973) 11-15. 

\bibitem{bbc2009} C. Badea, B. Beckermann and M. Crouzeix, Intersections of several disks of the Riemann sphere as $K$-spectral sets, {\em Comm. Pure Appl. Anal.} {\bf 8} (1) (2009) 37-54.


\bibitem{berg} C. A. Berger, A strange dilation theorem,
  {\em Amer. Math. Soc. Notices} {\bf 12} (1965)  590.
  
\bibitem{cartan} H. Cartan, Seminaire H. Cartan 1951/52, W. A. Benjamin, New York. 1967.
  
\bibitem{Crouzeix} M. Crouzeix, The annulus as a $K$-spectral set, {\em Commun. Pure Appl. Anal.} {\bf 11} (6) (2012) 2291-2303.

\bibitem{bfd1999} B. Delyon and F. Delyon, Generalizations of von Neumann's spectral sets and integral representations of operators, {\it Bull. Soc. Math, France} {\bf 127 } (1999) 25-41.
  
\bibitem{dw97}
M. A. Dritschel and H. J. Woerdeman, Model Theory and Linear Extreme Points in the Numerical Radius Unit Ball, {\it Memoirs Amer. Math. Soc.} {\bf 615} (1997) 1-62.

\bibitem{dp86}
R.~G. Douglas and V.~I. Paulsen,
\newblock Completely bounded maps and hypo-{D}irichlet algebras,
\newblock {\em Acta Sci. Math. (Szeged)},  {\bf 50} (1-2) (1986) 143-157.

\bibitem{GuRao97}  K. E. Gustafson and D. K. M. Rao, {\em Numerical Range. The Field of Values of Linear Operators and Matrices}, Springer, 1997.

\bibitem{krey} E. Kreyszig, {\em Advanced Engineering Mathematics}, John Wiley and Sons, Inc., New York,  8th edition, 1999.


\bibitem{szn-foi} B. Sz-Nagy and C. Foias,
 {\em Harmonic Analysis of Operators on {Hilbert} Space}, North Holland, Amsterdam, 1970.
 
\bibitem{kato} T. Kato, Some mapping theorems for the numerical range, {\it  Proc. Japan Acad.} {\bf 41} (1965) 652--655.

\bibitem{nai1} M. A. Naimark, 
  Positive definite operator functions on a commutative group,
 {\em Izv. Akad. Nauk SSSR Ser. Mat.} {\bf 7} (1943) 237-244.
 
\bibitem{nai2}
M. A. Naimark,
  On the representation of additive operator set functions,
 {\em Comptes Rendus (Doklady) Acad. Sci. URSS} {\bf 41} (1943) 359-361.
 
\bibitem{WuGau} P.Y.  Wu and H.-L. Gau, {\em Numerical Ranges of Hilbert Space Operators}, 
 Cambridge University Press, 2021.
 

\bibitem{sar} D. Sarason, On spectral sets having connected complement, 
{\em Acta Sei. Math. Szeged} {\bf 26} (1965) 289--299.

\bibitem{Shields} A. L. Shields, Weighted shift operators and and analytic function theory, in {\em Topics in Operator Theory. Mathematical Surveys and Monographs} {\bf 13} (1974) 49-128, American Math. Society, Providence, RI.
 
\bibitem{szn53} B. Sz\H{o}kefalvi-Nagy, Sur les contractions de l'espace de Hilbert,  {\em Acta  Sci. Math.} {\bf 15} (1953)  87-92.
 
\end{thebibliography}
\end{document}